\documentclass[11pt, reqno]{amsart}    
\usepackage{graphicx,epsfig,color}
\usepackage{amssymb,amsmath,amsthm}

\newtheorem{theorem}{Theorem}[section]
\newtheorem{lemma}[theorem]{Lemma}
\newtheorem{prop}[theorem]{Proposition}
\newtheorem{cor}[theorem]{Corollary}
\newtheorem{fact}[theorem]{Fact}
\newtheorem{claim}[theorem]{Claim}
\newtheorem{example}[theorem]{Example}

\def\ts{\textstyle}    
\def\Z{{\mathbb{Z}}} 
 
\def\cA{{\mathcal A}}
\def\cB{{\mathcal B}}
\def\cF{{\mathcal F}}
\def\cG{{\mathcal G}}
\def\cH{{\mathcal H}}

\def\cT{{\mathcal T}}
\def\tF{\tilde{\cF}}  
\def\hF{\dot{\cF}}  
\def\dhF{\ddot{\cF}}  
\def\bF{\bar{\cF}}  
\def\tA{\tilde{\cA}}  
\def\hA{\dot{\cA}}  
\def\dhA{\ddot{\cA}}  
\def\tB{\tilde{\cB}}  
\def\hB{\dot{\cB}}  
\def\dhB{\ddot{\cB}}  
\def\muA{\mu_p(\cA)}
\def\muB{\mu_p(\cB)}
\def\({\left(}
\def\){\right)}
\def\dual{{\rm dual}}
\def\first{{\rm first}_k}
\def\kmax{{\rm kmax}}
\def\shiftsto{\to}
\newcommand\bfrac[2]{\genfrac{(}{)}{}{}{#1}{#2}}    

\begin{document}
\title{An Erd\H os--Ko--Rado theorem for cross $t$-intersecting families}

\date{August 20, 2014}
\author{Peter Frankl}
\address{
Alfr\'ed R\'enyi Institute of Mathematics,
H-1364 Budapest, P.O.Box 127, Hungary
}
\email{peter.frankl@gmail.com}
\author{Sang June Lee}
\address{Department of Mathematics, Duksung Women's University, Seoul 132-714, South Korea}
\email{sanglee242@duksung.ac.kr, sjlee242@gmail.com}
\thanks{The second author was supported by Basic Science Research Program through the National Research Foundation of Korea (NRF) funded by the Ministry of Science, ICT \& Future Planning (NRF-2013R1A1A1059913); and also by Korea Institute for Advanced Study (KIAS) grant funded by the Korean government (MEST)
}
\author{Mark Siggers}
\address{College of Natural Sciences, Kyungpook National University,
          Daegu 702-701, South Korea}
\email{mhsiggers@knu.ac.kr}
\thanks{The third author was supported by the Korea NRF Basic Science Research Program (2014-06060000) funded by the Korean government (MEST);
                    and by the Kyungpook National University Research Fund (2012).}
\author{Norihide Tokushige}
\address{College of Education, Ryukyu University, Nishihara, Okinawa 903-0213, Japan}\email{hide@edu.u-ryukyu.ac.jp}
\thanks{The last author was supported by JSPS KAKENHI (20340022 and 25287031)}

\begin{abstract}
Two families $\cA$ and $\cB$, of $k$-subsets of an $n$-set, are 
{\em cross $t$-intersecting} if for every choice of subsets $A \in \cA$ and 
$B \in \cB$ we have $|A \cap B| \geq t$. 
We address the following conjectured cross $t$-intersecting version of 
the Erd\H os--Ko--Rado Theorem: For all $n \geq (t+1)(k-t+1)$ the maximum 
value of $|\cA||\cB|$ for two cross $t$-intersecting families 
$\cA, \cB \subset\binom{[n]}{k}$ is $\binom{n-t}{k-t}^2$. 
We verify this for all $t \geq 14$ except finitely many $n$ and $k$ for each
fixed $t$.
Further, we prove uniqueness and stability results in these cases, 
showing, for instance, that the families reaching this bound are unique 
up to isomorphism.  
We also consider a {\em $p$-weight} version of the problem,
which comes from the product measure on the power set of an $n$-set.

\end{abstract}

\maketitle

\section{ Introduction }
Let $[n]=\{1,2,\ldots,n\}$ and 
let $2^{[n]}$ denote the power set of $[n]$.
A family $\cF\subset 2^{[n]}$ is called
{\sl $t$-intersecting} if $|F\cap F'|\geq t$ for all $F,F'\in\cF$.
Let $\binom{[n]}k$ denote the set of all $k$-subsets of $[n]$.
A family in $\binom{[n]}k$ is called {\sl $k$-uniform}.
For example, 
$$\cF_0=\Big\{F\in\binom{[n]}k:[t]\subset F\Big\}$$
is a $k$-uniform $t$-intersecting family of size $|\cF_0|=\binom{n-t}{k-t}$.
Erd\H os, Ko and Rado \cite{EKR} proved that there exists some $n_0(k,t)$ 
such that if $n\geq n_0(k,t)$ and 
$\cF\subset\binom{[n]}k$ is $t$-intersecting, then $|\cF|\leq\binom{n-t}{k-t}$.
The smallest possible such $n_0(k,t)$ is $(t+1)(k-t+1)$.
This was proved by Frankl \cite{Fckt} for $t\geq 15$, and then completed
by Wilson \cite{W} for all $t$. These proofs are very different, the former
uses combinatorial tools while the later is based on the eigenvalue method.
If $n<(t+1)(k-t+1)$ then $\binom{n-t}{k-t}$ is no longer the maximum size.
In fact we can construct a $t$-intersecting family
$$
\cF_i^t(n,k):=\Big\{F\in\binom{[n]}k:|F\cap[t+2i]|\geq t+i\Big\}
$$
for $0\leq i\leq k-t$, and it can be shown that
$|\cF_0^t(n,k)|\geq|\cF_1^t(n,k)|$ iff $n\geq(t+1)(k-t+1)$.
Frankl conjectured in \cite{Fckt} that if $\cF\subset\binom{[n]}k$ is 
$t$-intersecting, then 
\begin{equation}\label{k-uniform-AK}
|\cF|\leq\max_i|\cF_i^t(n,k)|.
\end{equation}
This conjecture was proved partially
by Frankl and F\"uredi \cite{FF}, and then settled completely by Ahlswede and
Khachatrian \cite{AK1}. 
This result is one of the highlights of extremal set theory. 

We will use the proof technique used in \cite{FF}, 
in another direction, to deal with cross $t$-intersecting families. 
Two families $\cA,\cB\subset 2^{[n]}$
are called {\sl cross $t$-intersecting} if $|A\cap B|\geq t$ holds for all
$A\in\cA$ and $B\in\cB$.
Pyber \cite{P} considered the case $t=1$, and proved that if $n\geq 2k$ and
$\cA,\cB\subset\binom{[n]}k$ are cross $1$-intersecting, then
$|\cA||\cB|\leq\binom{n-1}{k-1}^2$.
It was then proved in \cite{MT} that if $n\geq\max\{2a,2b\}$, and
$\cA\subset\binom{[n]}a$ and $\cB\subset\binom{[n]}b$ are cross 1-intersecting,
then $|\cA||\cB|\leq\binom{n-1}{a-1}\binom{n-1}{b-1}$. See Borg~\cite{B} for a corresponding cross t-intersecting result
for $n>n_0(k,t)$. Gromov \cite{G} found 
an application of these inequalities to geometry.
For the general cross $t$-intersecting case, it is natural to expect that
if $n\geq(t+1)(k-t+1)$ and $\cA,\cB\subset\binom{[n]}k$ are cross 
$t$-intersecting, then $|\cA||\cB|\leq\binom{n-t}{k-t}^2$. This conjecture 
was verified for $n>2tk$ in \cite{T0} using a combinatorial approach, and
for $\frac kn<1-\frac 1{\sqrt[t]{2}}$, or more simply, $n>1.443(t+1)k$
in \cite{Teigen1} using the eigenvalue method. 
We also mention that Suda and Tanaka \cite{ST} obtained a similar result 
concerning cross $1$-intersecting families of vector subspaces 
based on semidefinite programming.

In this paper we prove the following
result which almost reaches the conjectured lower bound for $n$.
We say that two families $\cA$ and $\cB$ in $2^{[n]}$ are {\sl isomorphic} if
there is a permutation $f$ on $[n]$ such that $\cA=\{\{f(b):b\in B\}:B\in\cB\}$,
and in this case we write $\cA\cong \cB$.

\begin{theorem}\label{k-thm}
For every $k\geq t\geq 14$ and $n\geq (t+1)k$ we have the following.
If $\cA\subset\binom{[n]}k$ and $\cB\subset\binom{[n]}k$ are cross
$t$-intersecting, then
$$
|\cA||\cB|\leq\binom{n-t}{k-t}^2
$$
with equality holding iff $\cA=\cB\cong\cF_0^t(n,k)$.
\end{theorem}
The extremal configuration has a {\sl stability};
if $|\cA||\cB|$ is very close to $\binom{n-t}{k-t}^2$, 
then both families are very close to $\cF_0^t(n,k)$.
By saying $\cA$ is close to $\cF$ we mean that the symmetric difference 
$\cA\bigtriangleup\cF=(\cA\setminus \cF)\cup(\cF\setminus \cA)$ 
is of small size.
A family $\cA\subset 2^{[n]}$ is called {\sl shifted} if $(A\setminus\{j\})\cup\{i\}\in\cA$ whenever 
$1\leq i<j\leq n$, $A\in\cA$, and $A\cap\{i,j\}=\{j\}$. (We will explain more about shifting
operations in the next section.) 

\begin{theorem}\label{k-thm-stability}
For every $k\geq t\geq 14$, $\delta>0$, $n\geq(t+1+\delta)k$, and 
$\eta\in(0,1]$,
we have the following. If $\cA$ and $\cB$ are shifted
cross $t$-intersecting families in $\binom{[n]}k$, 
then one of the following holds.
\begin{enumerate}
 \item $\sqrt{|\cA||\cB|} < (1-\gamma\eta)\binom{n-t}{k-t}$, 
where $\gamma\in(0,1]$ depends only on $t$ and $\delta$.
 \item $|\cA\bigtriangleup\cF_0^t(n,k)|
+|\cB\bigtriangleup\cF_0^t(n,k)|<\eta\binom {n-t}{k-t}$.
\end{enumerate}
\end{theorem}

We also consider
the so-called {\sl $p$-weight version} or {\sl measure version}
(see e.g.\ \cite{FT2002,FTw,Fri}) of 
the above result concerning $k$-uniform families.
Let $p\in(0,1)$ be a fixed real, and let $\mu_p$ be the product measure
on $2^{[n]}$ 
defined by
$$\mu_p(F):=p^{|F|}(1-p)^{n-|F|}.$$
For a family $\cF\subset 2^{[n]}$ let us define its $p$-weight (or measure) by
$$\mu_p(\cF):=\sum_{F\in\cF}\mu_p(F).$$
Ahlswede and Khachatrian \cite{AK-p} proved that if $\cF\subset 2^{[n]}$ is
$t$-intersecting, then 
\begin{equation}\label{p-weight-AK}
\mu_p(\cF)\leq\max_i \mu_p(\cF_i^t(n))
\end{equation}
where
$$
\cF_i^t(n):=\{F\subset[n]:|F\cap[t+2i]|\geq t+i\}. 
$$
It is not difficult to derive \eqref{p-weight-AK} from \eqref{k-uniform-AK},
see \cite{Tuvsw, DS}.
In particular, if $p\leq \frac1{t+1}$ then 
$\max_i \mu_p(\cF_i^t(n))=\mu_p(\cF_0^t(n))=p^t$.
In \cite{Fri}, Friedgut gave a proof of \eqref{p-weight-AK}, in the case $p\leq \frac1{t+1}$,
using the eigenvalue method,  
which is the $p$-weight version of Wilson's proof \cite{W}. 
Friedgut's proof can easily
be extended to cross $t$-intersecting families if $p<\frac{0.69}{t+1}$
as in \cite{Teigen2}.  More precisely,
if $t\geq 1$, $p\leq 1-\frac 1{\sqrt[t]{2}}$,
 and two families $\cA,\cB\subset 2^{[n]}$
are cross $t$-intersecting, then we have $\mu_p(\cA)\mu_p(\cB)\leq p^{2t}$. 
In the present paper, we prove the same inequality for all 
$t\geq 14$ and $p\leq \frac1{t+1}$.

\begin{theorem}\label{p-thm}
For every $t\geq 14$, $n\geq t$, and $p$ with $0<p\leq\frac 1{t+1}$, 
we have the following. 
If $\cA\subset 2^{[n]}$ and $\cB\subset 2^{[n]}$ are cross $t$-intersecting,
then
$$
\mu_p(\cA)\mu_p(\cB) \leq p^{2t}.
$$
Equality holds iff either $\cA=\cB\cong\cF_0^t(n)$, or
$p=\frac 1{t+1}$ and $\cA=\cB\cong\cF_1^t(n)$.
\end{theorem}

We have the following stability result as well.
A family $\cG\subset 2^{[n]}$ is called {\sl inclusion maximal}
if $G\in\cG$ and $G\subset G'$ imply $G'\in\cG$.

\begin{theorem}\label{p-thm-stability}
For every $t\geq 14$, $n\geq t$, $\epsilon>0$, $\eta\in(0,1]$, and  
$p$ with $0<p\leq\frac 1{t+1}-\epsilon$,
we have the following. If $\cA$ and $\cB$ are shifted,
inclusion maximal cross $t$-intersecting
families in $2^{[n]}$, then one of the following holds.
\begin{enumerate}
 \item $\sqrt{\mu_p(\cA)\mu_p(\cB)} < (1-\gamma\eta)p^{t}$, 
where $\gamma\in(0,1]$ depends only on $t$ and $\epsilon$.
 \item $\mu_p(\cA\bigtriangleup\cF_0^t(n))
+\mu_p(\cB\bigtriangleup\cF_0^t(n))<\eta p^t$.
\end{enumerate}
\end{theorem}

In \cite{Fri}, Friedgut obtained similar stability results for 
(not necessarily shifted) $t$-intersecting families.
He used a result due to Kindler and Safra \cite{KS}, 
which states that Boolean functions whose Fourier transforms are concentrated 
on small sets, essentially depend on only a few variables.

We cannot replace condition (ii) of Theorem~\ref{k-thm-stability} with the
condition $\cA,\cB\subset\cF_0^t(n,k)$ which is sometimes sought
in such stability results. Indeed,  we can construct a shifted $t$-intersecting family 
$\cA\subset\binom{[n]}k$ such that 
$|\cA|=\binom{n-t}{k-t}(1-o(1))$ 
where $o(1) \to 0$
as $n,k\to\infty$ with $n>(t+1)k$, but $\cA\not\subset\cF$ for any
$\cF\cong\cF_0^t(n,k)$.
For this, let $\cT=\{F\in\cF_0^t(n,k):F\cap[t+1,k+1]=\emptyset\}$, 
$\cH=\{[k+1]\setminus\{i\}:1\leq i\leq t\}$, and let
$\cA=(\cF_0^t(n,k)\setminus\cT)\cup\cH$.
Then it is easy to see that $\cA$ fulfills the prescribed properties. 

Similarly, we cannot replace condition (ii) of Theorem~\ref{p-thm-stability}
with the condition $\cA,\cB\subset\cF_0^t(n)$.
In fact there is a $t$-intersecting family $\cG\subset 2^{[n]}$ 
such that $\mu_p(\cG)$ is arbitrarily close to $p^t$, 
but $\cG$ is not a subfamily of any isomorphic copy of $\cF_0^t(n)$.
For example, 
let $T=[t]$, $\cH=\{[n]\setminus\{i\}:1\leq i\leq t\}$, and let
$\cG=(\cF_0^t(n)\setminus\{T\})\cup\cH$.
Then $\cG$ is a shifted, inclusion maximal $t$-intersecting family
with $\cG\not\subset\cF_0^t(n)$ and $\cG\bigtriangleup\cF_0^t(n)=\cH \cup \{T\}$.
Moreover we have
$\mu_p(\cG)=p^t-p^tq^{n-t}+tp^{n-1}q=(1-o(1))p^t$ where $o(1) \to 0$ as $n\to\infty$
for fixed $t$ and $p$.

We conjecture that Theorem~\ref{k-thm} holds for all 
$n,k$, and $t$ such that $k\geq t\geq 1$ and $n>(t+1)(k-t+1)$,
and Theorem~\ref{p-thm} holds for all $t\geq 1$.
We also conjecture that Theorems~\ref{k-thm-stability} and \ref{p-thm-stability}
are valid for families that are not necessarily shifted as well.
We mention that one can show Theorem~\ref{k-thm} for $t\geq 14$, $k>k_0(t)$,
and $n>(t+1)(k-t+1)$ as well, see Theorem~\ref{k-thm k is big}
in section~\ref{sec:k-uni}.

The approach in this paper follows that used in \cite{FF,T0}.  
We relate subsets in the cross $t$-intersecting families with walks in the
plane. After a normalizing process called shifting, 
these families will have the property that 
the corresponding walks all hit certain lines.  
In the $p$-weight version, the measure of such families is bounded by the 
probability that a certain random walk hits the same lines.  
Results for $k$-uniform cross $t$-intersecting families can often be 
inferred by corresponding $p$-weight results applied 
to the families obtained by taking all supersets of the original 
$k$-uniform families, which will also be cross $t$-intersecting.
Indeed, using Theorem~\ref{p-thm-stability}
it is relatively easy to prove results similar to 
Theorems~\ref{k-thm} and \ref{k-thm-stability} 
but with somewhat weaker bounds for $n$ and $k$. 
However, to get our $k$-uniform results
in full strength, we need to prove them directly
instead of relying on our $p$-weight results.
Nevertheless understanding the proof of $p$-weight results is very helpful
for the proof of $k$-uniform results. They have a similar proof with corresponding steps,
though the actual computations appearing in the proof of the $p$-weight version are 
usually much easier than those of the $k$-uniform version.

The paper is organised as follows.  
In Section~\ref{sec:tools} we present tools that we will use
throughout the paper.  
In Section~\ref{sec:weighted} we prove Proposition~\ref{main-prop}, our 
main result about the $p$-weight version of the problem,  
from which Theorems~\ref{p-thm} and \ref{p-thm-stability} easily follow.
In Section~\ref{sec:k-uni} we prove Proposition~\ref{main-k-prop},
our main result about the $k$-uniform version of the problem, 
from which Theorems~\ref{k-thm} and \ref{k-thm-stability} follow.
In Section~\ref{sec:application} we present an application to families of
$t$-intersecting integer sequences.

\section{Tools}\label{sec:tools}

In this section we present some standard tools. 
The proofs are also standard (see, e.g. \cite{FF,FTw, Tbalaton}), but we include them 
for completeness.

Throughout this paper let $p\in(0,1)$ be a real number, let $q=1-p$, and let $\alpha=p/q$.
The {\sl walk} associated to a set $F\subset[n]$ is an $n$-step walk on the 
integer grid $\Z^2$ starting at the origin $(0,0)$ whose $i$-th step is 
{\sl up} (going from $(x,y)$ to $(x,y+1)$) if $i \in F$, 
and is {\sl right} (going from $(x,y)$ to $(x+1, y)$) 
if $i \not\in F$. 
We thus refer to $F \in 2^{[n]}$ as either a set or a walk, 
depending on which point of view is more convenient. 
Correspondingly, consider an $n$-step random walk $W_{n,p}$ 
whose $i$-th step is a random variable, independent of other steps, going `up' with probability $p$
and `right' otherwise.  
Since $\mu_p$ is a probability measure on $2^{[n]}$,
the $p$-weight of a family $\mu_p(\cF)$, where $\cF\subset 2^{[n]}$ 
consists of all walks that satisfy a given property P,
is exactly the probability that $W_{n,p}$ satisfies P.

\begin{example}\label{ex:hitpoint}
The $p$-weight of the family of all walks in $2^{[n]}$ that hit the point 
$(0,t)$ is the probability that $W_{n,p}$ hits $(0,t)$, which is $p^t$. 
The $p$-weight of the family of all walks in $2^{[n]}$ that hit $(1,t)$ 
but not $(0,t)$ is $tp^tq$. 
Indeed for a walk to hit $(1,t)$ but not $(0,t)$, it must move up $t-1$ of 
its first $t$ steps, this can be done in $t$ ways, 
and then must move up on the $(t+1)$-th step. So 
the probability is $\binom{t}{1}p^{t-1}q \cdot p$, as needed.      
\end{example}

\begin{lemma}\label{lem:hitline}
Let $\cF\subset 2^{[n]}$, and let $t$ be a positive integer.
\begin{enumerate}
\item If all walks in $\cF$ hit the line $y=x+t$, 
then $\mu_p(\cF)\leq\alpha^t$. 
\item For every $\epsilon$ there is an $n_0$ such that if $n>n_0$ and
no walk in $\cF$ hits the line $y=x+t$, 
then $\mu_p(\cF)<1-\alpha^t+\epsilon$.
\item If all walks in $\cF$ hit the line $y=x+t$ at least twice, 
but do not hit the line $y=x+(t+1)$, then $\mu_p(\cF)\leq\alpha^{t+1}$.
\end{enumerate}
\end{lemma}

\begin{proof}
We notice that, for fixed $p$, the probability 
$P_n:={\rm Prob}(W_{n,p}\text{ hits }y=x+t)$ is monotone increasing 
and bounded, 
and hence $\lim_{n\to\infty}P_n$ exists. In fact this limit is known to be
exactly $\alpha^t=(p/q)^t$, see e.g., \cite{Tbalaton}. This gives (i) and (ii).

There is an injection from (I) the family of walks that hit the line $y=x+t$ 
at least twice but do not hit $y = x + (t+1)$ to 
(II) the family of walks that hit $y = x + (t+1)$.   
Indeed for a walk $F$ in (I) that hits $y = x + t$ for the first time
at $(x_1, x_1+t)$ and for the second time at $(x_2, x_2+t)$, 
we get a walk in (II) by reflecting the portion of $F$ between $(x_1, x_1+t)$ 
and $(x_2, x_2+t)$ across the line $y = x + t$.
Further, these walks have the same $p$-weight. Thus we have (iii).
\end{proof}
 
For $1\leq i<j\leq n$ we define the {\sl shifting operation}
$s_{ij}:2^{[n]}\to 2^{[n]}$ by 
$$s_{ij}(\cF):=\{s_{ij}(F):F\in\cF\}$$ 
where $\cF\subset 2^{[n]}$ and
\[ 
s_{ij}(F) := \begin{cases} 
      (F \setminus \{j\})\cup \{i\}  & \text{if } F\cap\{i,j\}=\{j\} 
       \text{ and } (F \setminus \{j\})\cup \{i\}\not\in\cF,\\
    F & \text{otherwise.} \\
              \end{cases}  
\]
A family $\cF$ is called {\sl shifted} if $s_{ij}(\cF) = \cF$ 
for all $1\leq i < j \leq n$. 
Here we list some basic properties concerning shifting operations.

\begin{lemma}\label{lem:shifting}
Let $1\leq i<j\leq n$ and let $\cF,\cG\subset 2^{[n]}$.
\begin{enumerate}
\item Shifting operations preserve the $p$-weight of a family, that is,
$\mu_p(s_{ij}(\cG))=\mu_p(\cG)$.
\item If $\cG_1$ and $\cG_2$ in $2^{[n]}$ are cross $t$-intersecting families, 
then $s_{ij}(\cF)$ and $s_{ij}(\cG)$ are cross $t$-intersecting 
families as well.
\item For a pair of families we can always obtain a pair of 
shifted families by repeatedly shifting families simultaneously 
finitely many times.
\item If $\cG$ is inclusion maximal, and $s_{ij}(\cG)=\cF_\ell^t(n)$, 
then $\cG\cong\cF_\ell^t(n)$ for $\ell=0,1$.
\end{enumerate}
\end{lemma}

\begin{proof}
Since $|s_{ij}(G)|=|G|$ for $G\subset[n]$ we have 
$\mu_p(s_{ij}(G))=p^{|s_{ij}(G)|}q^{n-|s_{ij}(G)|}=p^{|G|}q^{n-|G|}=\mu_p(G)$. 
Thus
$\mu_p(s_{ij}(\cG))=\sum_{G\in\cG}\mu_p(s_{ij}(G))=
\sum_{G\in\cG}\mu_p(G)=\mu_p(\cG)$. This gives (i).

Let $\cF'=s_{ij}(\cF)$ and $\cG'=s_{ij}(\cG)$.  
Suppose that $\cF$ and $\cG$ are cross $t$-intersecting, but
$\cF'$ and $\cG'$ are not. Then there are 
$F\in\cF$ and $G\in\cG$ such that $|F\cap G|\geq t$ but
$|F'\cap G'|<t$, where $F'=s_{ij}(F)$ and $G'=s_{ij}(G)$.
Consider the case when $F\cap\{i,j\}=\{j\}$ and $G\cap\{i,j\}=\{j\}$.
 (The other cases are ruled out easily.) By symmetry we may assume that
$F'\cap\{i,j\}=\{j\}$ and $G'\cap\{i,j\}=\{i\}$. This means that
$F'=F$ and this happens because 
$F_1:=(F\setminus\{j\})\cup\{i\}$ is already in $\cF$. Then
$|F_1\cap G|=|F'\cap G'|<t$, which contradicts the cross 
$t$-intersecting property of $\cF$ and $\cG$. This shows (ii).

Next we show (iii).
Let $\cG_1,\cG_2\subset 2^{[n]}$. Suppose that at least one of these
families, say, $\cG_1$ is not shifted. Then there is a shifting $s_{ij}$
such that $s_{ij}(\cG_1)\neq \cG_1$. 
Let $f(\cG)$ be the total sum of elements in the subsets of $\cG$, that is,
$f(\cG):=\sum_{G\in\cG}\sum_{x\in G}x$. Then
$f(s_{ij}(\cG_1))\leq f(\cG_1)-j+i\leq f(\cG_1)-1$. 
Namely, we can decrease the value $f(\cG_1)+f(\cG_2)$ at least 1 by applying
a shifting operation unless both of the families are already shifted.
On the other hand $f(\cG_1)+f(\cG_2)\geq 0$ for all $\cG_1,\cG_2$.
Thus we get (iii).

Finally we prove (iv).
Let $\cG'=s_{ij}(\cG)=\cF_\ell^t(n)$. 
Observe that $\binom{[t+2\ell]}{t+\ell}$ is a `generating set' of
$\cG'$, namely, 
\[
\cG'=\{G\subset[n]:F\subset G\text{ for some }F\in\binom{[t+2\ell]}{t+\ell}\}. 
\]
First let $\ell = 0$. Then since $\cG'=\cF_0^t(n)$ we have $[t]\in\cG'$.
Thus $\cG$ must contain some $t$-element set $G_0$, and so as $\cG$ is inclusion
maximal, $\cG$ contains $\{G\subset[n]:G_0\subset G\}\cong\cF_0^t(n)$.  
On the other hand, by (i) and our assumption, we have
$\mu_p(\cG)=\mu_p(s_{ij}(\cG))=\mu_p(\cF_0^t(n))$.
Thus we indeed have $\cG\cong\cF_0^t(n)$. 

Next let $\ell=1$. 
If $|\{i,j\}\cap[t+2]|=0$ or $2$, then it is easy to see that 
$s_{ij}(\cG)=\cG$ and we are done, so we may assume that $i=t+2$ and $j=t+3$.
Since $\cG'=\cF_1^t(n)$ we have 
$\binom{[t+2]}{t+1}\subset\cG'$. 
If $\binom{[t+2]}{t+1}\subset\cG$ then $\cG=\cF_1^t(n)$, too.
If $\binom{[t+2]}{t+1}\not\subset\cG$ then there is some 
$G'\in\binom{[t+2]}{t+1}$ such that $G'\not\in\cG$. In this case we have
$G'=A\cup\{t+2\}\in\cG'\setminus\cG$ for some $A\in\binom{[t+1]}{t}$,
and $G=A\cup\{t+3\}\in\cG\setminus\cG'$. This means $s_{ij}(G)=G'$.
For $x\in\{t+2,t+3\}$ let $\cG(x):=\{A\in\binom{[t+1]}{t}:A\cup\{x\}\in\cG\}$.
Then $\cG(t+2)\cup\cG(t+3)$ is a partition of $\binom{[t+1]}t$.
It follows from $G\in\cG(t+3)$ that $\cG(t+3)\neq\emptyset$.
If there is some $G''\in\cG(t+2)$ then $|G\cap G''|\geq t$ implies that
$G''=G'$, which is a contradiction because $G'\not\in\cG$. 
Thus $\cG(t+2)=\emptyset$ must hold. Consequently we have
$\cG=\{G\subset [n]:|G\cap T|\geq t+1\}$ where $T=[t+3]\setminus\{t+2\}$,
and $\cG\cong\cF_1^t(n)$.
\end{proof}

The following two simple facts are used only to prove Lemma~\ref{k-shifting} 
below. For $n>2k$ we define a Kneser graph $K(n,k)=(V,E)$ on the vertex set
$V=\binom{[n]}k$ by $(F,F')\in E$ iff $F\cap F'=\emptyset$.
\begin{fact}\label{kneser}
If $n>2k$, then the Kneser graph $K(n,k)$ is connected and non-bipartite.
\end{fact}
\begin{proof}
We use Katona's cyclic permutation method~\cite{Ka}. Observe that
$K(2k+1,k)$ contains $C_{2k+1}$ (a cycle of length $2k+1$). To see this,
let $F_i=\{i,i+1,\ldots,i+k-1\}$ (indices are read modulo $2k+1$), then
$F_0,F_{k}, F_{2k},\ldots,F_{(k-1)k}$ give the cycle. Moreover any two vertices 
$F,F'\subset[2k+1]$ are on some $C_{2k+1}$, because one can choose a
cyclic ordering $i_1,i_2,\ldots,i_{2k+1}$ such that both $F$ and $F'$ consist of 
consecutive elements in this ordering. Since $K(n,k)$ ($n>2k$) contains $K(2k+1,k)$ as an induced
subgraph, it follows that $K(n,k)$ is connected and non-bipartite.
\end{proof}

For two graphs $G$ and $H$ we define the direct product $G\otimes H=(V,E)$
on $V=V(G)\times V(H)$ by $((u,v),(u',v'))\in E$ iff $uu'\in E(G)$ and
$vv'\in E(H)$.

\begin{fact}\label{connected}
Let $G$ and $H$ be connected and non-bipartite graphs.
\begin{enumerate}
\item $G\otimes H$ is connected and non-bipartite.
\item $G\otimes K_2$ is connected,
where $K_2$ is the complete graph of order 2.
\end{enumerate}
\end{fact}

\begin{proof}
By a closed trail of length $n$ in $G$ we mean a sequence of vertices
$x_0x_1\ldots x_{n-1}x_0$ such that $x_ix_{i+1}\in E(G)$ for all $i$
(indices are read modulo $n$). Since $G$ is connected and non-bipartite,
one can find a closed trail of odd length containing any given two vertices.
Now let $(x,y),(x',y')\in G\otimes H$ be given. 
Choose a closed trail $x_0x_1\ldots x_{n-1}x_0$ containing $x,x'$ in $G$,
and a closed trail $y_0y_1\ldots y_{m-1}y_0$ containing $y,y'$ in $H$,
where both $n$ and $m$ are odd.
Then $(x_i,y_i)$, $i=0,1,\ldots,mn-1$, give a closed trail of length
$mn$ in $G\otimes H$, where indices of $x_i$ are read modulo $n$, while
indices of $y_i$ are read modulo $m$. This closed odd trail contains
both $(x,y)$ and $(x',y')$, so there is a path from $(x,y)$ to $(x',y')$
and there is an odd cycle in this closed trail. Thus we get (i).

One can prove (ii) directly, but this is a special case of Weichsel's
result \cite{Weichsel} which states that if $G$ and $H$ are connected, then
$G\otimes H$ is connected iff $G$ or $H$ contains an odd cycle.
\end{proof}

\begin{lemma}\label{k-shifting}
Let $k-t\geq \ell\geq 0$ and $\cF:=\cF_\ell^t(n,k)$.
\begin{enumerate}
\item If $\cF$ and $\cB\subset\binom{[n]}k$ 
are cross $t$-intersecting, and $|\cF|=|\cB|$, then $\cF=\cB$.
\item 
Let $n\geq 2k-t+2$ and $t\geq 2$.
If $\cA$ and $\cB$ are cross $t$-intersecting families in $\binom{[n]}k$,
and $s_{ij}(\cA)=s_{ij}(\cB)=\cF$, then $\cA=\cB\cong\cF$.
\end{enumerate}
\end{lemma}

\begin{proof}

To prove (i) we notice that $\cF$ is a maximal $t$-intersecting family 
in the sense that adding any $k$-subset (not contained in $\cF$)
to $\cF$ would destroy the $t$-intersecting property. 
Since $\cF$ and $\cB$ are cross $t$-intersecting,
for any $B\in\cB$, $\cF\cup\{B\}$ is still $t$-intersecting. 
This with the maximality of $\cF$ forces $B\in\cF$, namely, $\cB\subset\cF$. 
Then $|\cF|=|\cB|$ gives $\cF=\cB$.

Next we prove (ii) following \cite{AK1}.
For $1\leq i<j\leq n$ and a family $\cG\subset 2^{[n]}
$, let
\begin{align*}
\cG[\bar ij]&:=\{G\in\cG:i\not\in G,\,j\in G,\,
(G\cup\{i\})\setminus\{j\}\not\in\cG\},\\
 \cG[i\bar j]&:=\{G\in\cG:j\not\in G,\,i\in G,\,
(G\cup\{j\})\setminus\{i\}\not\in\cG\},
\end{align*}
and let $\tilde \cG$ be the family obtained from $\cG$ by exchanging
the coordinates $i$ and $j$. 
We list some basic properties about these families.
\begin{itemize}
\item By definition, $\cG\cong\tilde\cG$. 
Also $\cG[i\bar j]\cong\tilde\cG[\bar ij]$,
$\cG[\bar ij]\cong\tilde\cG[i\bar j]$.
\item It follows that $\cG\setminus s_{ij}(\cG)=\cG[\bar ij]$ and
$s_{ij}(\cG)\setminus \cG=\tilde\cG[i\bar j]$.
\item If $\cG[\bar ij]=\emptyset$, then $s_{ij}(\cG)=\cG$.
\item It follows that $\cG\setminus\tilde\cG=\cG[\bar ij]\cup\cG[i\bar j]$ 
and $\tilde\cG\setminus\cG=\tilde\cG[\bar ij]\cup\tilde\cG[i\bar j]$. 
\item If $\cG[i\bar j]=\emptyset$, then 
$\cG\cap\tilde\cG=\cG\cap s_{ij}(\cG)=\cG\setminus\cG[\bar ij]
=\tilde\cG\setminus\tilde\cG[i\bar j]$.
\item If $\cG[i\bar j]=\emptyset$, then $s_{ij}(\cG)=\tilde\cG$.
In fact, if $\cG[i\bar j]=\emptyset$, then $\tilde\cG[\bar ij]=\emptyset$, and
\[
\tilde\cG=(\cG\cap\tilde\cG)\cup\tilde\cG[i\bar j]
=(\cG\cap s_{ij}(\cG))\cup\tilde\cG[i\bar j]=s_{ij}(\cG).
\]
\end{itemize}

Now we assume that $\cA$ and $\cB$ are cross $t$-intersecting, and
$s_{ij}(\cA)=s_{ij}(\cB)=\cF$. Then clearly $|\cA|=|\cB|$.
We will show that $\cA=\cB=\cF$ or $\cA=\cB=\tilde\cF$.

If $\cA[\bar ij]=\emptyset$, then $s_{ij}(\cA)=\cA$. 
Thus $\cA=\cF$, and (i) gives that $\cA=\cB=\cF$, as desired.
Similarily, if $\cA[i\bar j]=\emptyset$, then $s_{ij}(\cA)=\tilde\cA$.
Thus $\tilde\cA=\cF$, and (i) gives that $\tilde\cA=\tilde\cB=\cF$, 
or equivalently, $\cA=\cB=\tilde\cF$.

Thus we may assume that $\cA[\bar ij]\neq\emptyset$ and
$\cA[i\bar j]\neq\emptyset$. By the same reasoning, we may assume that
$\cB[\bar ij]\neq\emptyset$ and $\cB[i\bar j]\neq\emptyset$.
We will show that this is impossible.
Without loss of generality we may also assume that $i=t+2\ell$ and $j=i+1$.

Note that $F\in\cF[i\bar j]$ iff $|F\cap[t+2\ell-1]|=t+\ell-1$. 
Keeping this in mind, let
\[
 \cH:=\{H\in\binom{[n]\setminus\{i,j\}}{k-1}:|H\cap[t+2\ell-1]|=t+\ell-1\}.
\]
For every $H\in\cH$ we have $H\cup\{i\}\in\cF$ and $H\cup\{j\}\not\in\cF$.
Since $\cF=s_{ij}(\cA)$ it follows that
\begin{equation}\label{bijection}
\text{
either $H\cup\{i\}\in\cA$ or $H\cup\{j\}\in\cA$ (but not both). 
}
\end{equation}
(In fact, if both hold, then 
$s_{ij}(H\cup\{j\})=H\cup\{j\}\in s_{ij}(\cA)=\cF$, a contradiction.)
If $A\in\cA_{ij}:=\cA[i\bar j]\cup\cA[\bar ij]$, 
then $|A\cap[t+2\ell-1]|=t+\ell-1$.
(In fact if $A\in\cA[i\bar j]$, then 
$A':=(A\cup\{j\})\setminus\{i\}\not\in\cA$, 
which means that there is some $B\in\cB$ such that $|A\cap B|\geq t$ but
$|A'\cap B|<t$, and this happens only when 
$|A'\cap[t+2\ell-1]|=|B\cap[t+2\ell-1]|=t+\ell-1$.)
Thus \eqref{bijection} defines a bijection $f:\cH\to\cA_{ij}$.
Similarly we obtain a bijection $g:\cH\to\cB_{ij}$, where 
$\cB_{ij}:=\cB[i\bar j]\cup\cB[\bar ij]$.

Here we construct a bipartite graph $G=(V_\cA\cup V_\cB,E)$, where
both $V_\cA$ and $V_\cB$ are copies of $\cH$, 
and $(H_A,H_B)\in E$ iff $|H_A\cap H_B|=t-1$.
We divide $V_\cA$ into $V_{\cA[i\bar j]}$ and $V_{\cA[\bar ij]}$ 
according to whether $f(H)\in\cA[i\bar j]$ or $f(H)\in\cA[\bar ij]$.
In the same way, we also get the partition 
$V_\cB=V_{\cB[i\bar j]}\cup V_{\cB[\bar ij]}$ using $g$.
It then follows from the cross $t$-intersecting property that
there are no edges between $\cA[\bar ij]$ and $\cB[i\bar j]$, and 
no edges between $\cA[i\bar j]$ and $\cB[\bar ij]$. 
Recall that none of $\cA[i\bar j]$, $\cA[\bar ij]$, 
$\cB[i\bar j]$, and $\cB[\bar ij]$ are empty. 
Thus the graph $G$ is disconnected.

Let $G_0=(V_0,E_0)$ be a graph such that $V_0=\cH$
and $(F_0,F_0')\in E'$ iff $|F_0\cap F_0'|=t-1$.
If $n\geq 2k-t+2$ and $t\geq 2$, then it is readily seen that 
$G_0$ is isomorphic to $K(n_1,k_1)\otimes K(n_2,k_2)$, where
\begin{align*}
 n_1&=t+2\ell-1,& k_1&=(t+2\ell-1)-(t+\ell-1)=\ell,\\
n_2&=(n-2)-(t+2\ell-1),& k_2&=(k-1)-(t+\ell-1)=k-t-\ell. 
\end{align*}
(We used $n\geq 2k-t+2$ and $t\geq 2$ to ensure that 
$n_1>2k_1$ and $n_2>2k_2$.)
So it follows from Fact~\ref{kneser} and Fact~\ref{connected} that
$G_0$ is connected and non-bipartite.
By definition, $G$ is isomorphic to $G_0\otimes K_2$, and
by Fact~\ref{connected}, $G$ is connected. This is a contradiction.
\end{proof}

Similarly one can show the following, which can be used as an alternative
to Lemma~\ref{lem:shifting} (iv).
 (For a proof we use that a graph $G=(V,E)$, 
where $V=2^{[n]}$ and $(F,F')\in E$ iff $F\cap F'=\emptyset$, is connected
and non-bipartite for $n\geq 3$.)

\begin{lemma}\label{non-unif-shifting}
Let $n\geq 3$, $t\geq 2$, and $k-t\geq\ell\geq 0$.
If $\cA$ and $\cB$ are cross $t$-intersecting families in $2^{[n]}$,
and $s_{ij}(\cA)=s_{ij}(\cB)=\cF_\ell^t(n)$, then $\cA=\cB\cong\cF_\ell^t(n)$.
\end{lemma}

For $A\subset[n]$ let $(A)_i$ denote the $i$-th element of $A$, where
$(A)_1<(A)_2<\cdots$. For $A,B\subset [n]$, we say {\em $A$ shifts to $B$}, and 
write  
\[
A \shiftsto B  
\]
if $|A|\leq |B|$ and $(A)_i\geq (B)_i$ for all $i\leq |A|$.
E.g., $\{2,4,6,8\}\shiftsto\{1,2,4,8,9\}$.

We list some easy facts below, which we will use without referring to explicitly.

\begin{fact}\label{fact:toolfact1}
Let $\cA\subset 2^{[n]}$ be shifted.
\begin{enumerate} 
 \item If $A \in \cA$, $A \shiftsto A'$, and         
$|A|=|A'|$, then $A'\in\cA$.
\item If $\cA$ is inclusion maximal, 
$A \in \cA$, and $A \shiftsto A''$, then $A''\in \cA$.  
\end{enumerate}
\end{fact}

Let $t\in[n]$ and $A\subset[n]$.
We define the {\sl dual} of $A$ with respect to $t$ by
\[
\dual_t(A):=[(A)_t-1]\cup([n]\setminus A). 
\]

Clearly we have $|A\cap\dual_t(A)|=t-1$.  
The following simple fact is very useful, and we will use it 
for some particular choices of $A$.

\begin{fact}\label{fact:toolfact2}
If $\cA$ and $\cB$ are cross $t$-intersecting and
$A\in\cA$, then $\dual_t(A)\not\in\cB$.
\end{fact}

For $\cF\subset 2^{[n]}$ let $\lambda(\cF)$ be the maximum $\lambda$ 
such that all walks in $\cF$ hit the line $y = x + \lambda$.
Let
\[
F_{[u]}:=[u]\cup\{u+2i\in[n]:i\geq 1\}
\]
be the ``maximal'' walk that does not hit $y=x+(u+1)$.
If $u\leq t$, then 
\[
\dual_t(F_{[u]})=F_{[2t-u-1]}.
\]
Note also that if $F\subset[n]$ does not hit the line $y=x+(u+1)$, 
then $F\shiftsto F_{[u]}$. In particular, we have the following.

\begin{fact}\label{fact:toolfact3}
For a shifted, inclusion maximal family $\cF\subset 2^{[n]}$, if 
$F_{[u]}\not\in\cF$ then $\lambda(\cF)\geq u+1$.
\end{fact}

\begin{lemma}\label{eq:>=i+t}
Let $\cA$ and $\cB$ be shifted cross $t$-intersecting families in $2^{[n]}$.
\begin{enumerate}
 \item For every $A\in\cA$ and $B\in\cB$ there is some $i$ such that
$|A\cap[i]| + |B\cap[i]| \geq i+t$.
\item Suppose further that $\cA$ and $\cB$ are inclusion maximal 
cross $t$-intersecting families.
Then $\lambda(\cA) + \lambda(\cB) \geq 2t$.
\end{enumerate}
\end{lemma}

\begin{proof}
Suppose the contrary to (i). Choose a pair of counterexamples $A\in \cA$, $B\in \cB$ so that
$|A\cap B|$ is minimal. 
Let $j=(A\cap B)_t$. (Recall that this is the $t$-th element of $A\cap B$.)
Then we have
$|A\cap[j]|+|B\cap[j]|<t+j=|A\cap B\cap[j]|+|[j]|$,
which is equivalent to
$|(A\cup B)\cap [j]|<|[j]|$.
Thus we can find some $i$ with $1\leq i<j$ such that $i\not\in A\cup B$, where $i\neq j$ follows from $j=(A\cap B)_t$.
Since $\cB$ is shifted, we have $B':=(B-\{j\})\cup\{i\}\in\cB$. Then
$|B\cap[j]|=|B'\cap[j]|$, and so $A$ and $B'$ are also counterexamples.
But we get $|A\cap B'|<|A\cap B|$, which contradicts the minimality.
This gives (i).

Let $u=\lambda(\cA)\leq\lambda(\cB)$. We may assume that $u< t$.
Since $\cA$ is inclusion maximal and $u=\lambda(\cA)$ we have that 
$F_{[u]}\in\cA$. Then cross $t$-intersecting property yields 
$\dual_t(F_{[u]})=F_{[2t-u-1]}\not\in\cB$. Since $\cB$ is also
inclusion maximal this with Fact~\ref{fact:toolfact3} gives 
$\lambda(\cB)\geq 2t-u$, as desired.
\end{proof}

Cross $t$-intersecting families have the following monotone property.
Let $f(n)$ be the maximum of $\muA\muB$ where
$\cA$ and $\cB$ are cross $t$-intersecting families in $2^{[n]}$.
\begin{lemma}\label{monotone}
$f(n)\leq f(n+1)$.
\end{lemma}
\begin{proof}
 Suppose that $\cA$ and $\cB$ are cross $t$-intersecting families in $2^{[n]}$
with $f(n)=\muA\muB$. Let $\cA'=\cA\cup\cA''\subset 2^{[n+1]}$ where
$\cA''=\{A\cup\{n+1\}:A\in\cA\}$. We write $\mu_p^{n}$ for the $p$-weight
to emphasize the size of the ground set. Then we have
$\mu_p^{n+1}(\cA')=q\mu_p^{n}(\cA)+p\mu_p^{n}(\cA)=\mu_p^{n}(\cA)$.
Similarly, letting $\cB'=\cB\cup\{B\cup\{n+1\}:B\in\cB\}$, we also have
$\mu_p^{n+1}(\cB')=\mu_p^{n}(\cB)$ and thus
$\mu_p^{n+1}(\cA')\mu_p^{n+1}(\cB')=\mu_p^{n}(\cA)\mu_p^{n}(\cB)$.
Since $\cA'$ and $\cB'$ are cross $t$-intersecting families in $2^{[n+1]}$, 
we have $f(n)\leq f(n+1)$.
\end{proof}

Many of the above results have natural $k$-uniform versions. For example
Lemma~\ref{lem:hitline} can be transformed as follows.

\begin{lemma}\label{lem:k-hitline}
Let $x_0,y_0,c$ be integers with $0<c<y_0<x_0+c$.
\begin{enumerate}
\item 
The number of walks from $(0,0)$ to $(x_0,y_0)$ which hit the line $y=x+c$ is 
$\binom{x_0+y_0}{y_0-c}$. In particular,
if all walks in $\cF\subset\binom{[n]}k$ hit the line $y=x+c$, 
then $|\cF|\leq\binom n{k-c}$. 
\item The number of walks from
$(0,0)$ to $(x_0,y_0)$ which do not hit the line $y=x+c$  is 
$\binom{x_0+y_0}{x_0}-\binom{x_0+y_0}{y_0-c}$.
In particular, if no walk in $\cF\subset\binom{[n]}k$ hits the line $y=x+c$, 
then $|\cF|\leq\binom nk-\binom n{k-c}$.
\item If all walks in $\cF$ hit the line $y=x+c$ at least twice, 
but do not hit the line $y=x+(c+1)$, then $|\cF|\leq \binom n{k-c-1}$.
\end{enumerate}
\end{lemma}

\begin{proof}
The walks from $O=(0,0)$ to $P=(x_0,y_0)$ 
that hit the line $L:y=x+c$ are in bijection with the walks from 
$(-c,c)$ to $P$; this is seen by reflecting the part of the walk, 
from $O$ to the first hitting point on the line $L$, in this line, 
and the number of such lines is $\binom{x_0+y_0}{y_0-c}$. This gives (i).
If $\cF\subset\binom{[n]}k$, then notice that $x_0=n-k$ and $y_0=k$.

There are $\binom{x_0+y_0}{x_0}$ walks from $(0,0)$ to $(x_0,y_0)$,
and $\binom{x_0+y_0}{y_0-c}$ of them hit the given line by (i). This gives (ii).

Consider walks from $(0,0)$ to $(n-k,k)$. Then, as in the proof of
(iii) of Lemma~\ref{lem:hitline}, 
there is an injection from (I) the family of walks that hit the line $y=x+c$ 
at least twice but do not hit $y = x + (c+1)$ to 
(II) the family of walks that hit $y = x + (c+1)$. Thus (iii) follows
from (i).
\end{proof}

In the $k$-uniform setting, observe  that  if $A,B \in \binom{[n]}{k}$, 
$A \shiftsto B$ simply means that $(A)_i\geq (B)_i$ for all $i\leq k$. 
So Fact~\ref{fact:toolfact1} reads as follows.
\begin{fact}
Let $\cA\subset\binom{[n]}k$ be shifted. If $A \in \cA$, and $A \shiftsto A'$, 
then $A'\in\cA$.
\end{fact}

 For $A=\{x_1,x_2,\ldots,x_k,\ldots\}$ with $|A|\geq k$ and $x_1<x_2<\cdots$,
 let $\first(A)$ be the first $k$ elements of $A$, that is,
\[
\first(A):=\{x_1,x_2,\ldots,x_k\}. 
\]
For $t\leq k \leq n$ we define the {\sl dual}  of $A\in\binom{[n]}k$ 
with respect to $t$ and $k$ by
\[
\dual_{t}^{(k)}(A):= \first(\dual_t(A)).
\]
Again, we clearly have that  $|A\cap\dual_{t}^{(k)}(A)|=t-1$, 
and Fact~\ref{fact:toolfact2} reads as follows.
\begin{fact}
If $\cA$ and $\cB$ are $k$-uniform cross $t$-intersecting families and
$A\in\cA$, then $\dual_t^{(k)}(A)\not\in\cB$.
\end{fact}
Let $F_{[u]}^{(k)} := \first(F_{[u]})$.
If $u\leq t$ and $2t-u-1\leq k$, then
\[
 \dual_t^{(k)}(F_{[u]}^{(k)})=F^{(k)}_{[2t-u-1]}.
\]
If $F\in\binom{[n]}k$ does not hit the line $y=x+(u+1)$, 
then $F\shiftsto F_{[u]}^{(k)}$. 
\begin{fact}
For a shifted family $\cF\subset \binom{[n]}k$, 
if $F_{[u]}^{(k)}\not\in\cF$ then $\lambda(\cF)\geq u+1$.
\end{fact}

We remark that (i) of Lemma~\ref{eq:>=i+t} is valid for 
$\cA,\cB\subset\binom{[n]}k$ as well. As for (ii) we assume that
both families are non-empty instead of inclusion maximal, as follows.

\begin{lemma}\label{eq:>=2t}
Let $\cA$ and $\cB$ be shifted cross $t$-intersecting families in 
$\binom{[n]}k$. 
If $|\cA||\cB|>0$, then $\lambda(\cA) + \lambda(\cB) \geq 2t$.
\end{lemma}
\begin{proof}
Let $u=\lambda(\cA)\leq\lambda(\cB)$. We may assume that $u< t$.
Since $\cA$ is shifted and $u=\lambda(\cA)$ we have that 
$F_{[u]}^{(k)}\in\cA$. Then cross $t$-intersecting property yields 
$\dual_t^{(k)}(F_{[u]}^{(k)})=F_{[2t-u-1]}^{(k)}\not\in\cB$. 
If $2t-u-1\leq k$, then this gives $\lambda(\cB)\geq 2t-u$, as desired.
If $2t-u-1>k$, then 
\[
|[k]\cap F_{[u]}^{(k)}|\leq u+\frac{k-u}2<u+\frac{2t-2u-1}2<t, 
\]
and $[k]\not\in\cB$. Since $\cB$ is shifted, this means 
$\cB=\emptyset$. But this contradicts our assumption $|\cA||\cB|>0$.
\end{proof}

\section{Results about weighted families}\label{sec:weighted}
In this section we prove the following main result from which 
Theorems~\ref{p-thm}, \ref{p-thm-stability}, and \ref{int-sec-thm} will follow. 
 
\begin{prop}\label{main-prop}
For every $t\geq 14$, $n\geq t$, $\eta\in(0,1]$, and  
$p$ with $0<p\leq\frac 1{t+1}$,
we have the following. If $\cA$ and $\cB$ are shifted,
inclusion maximal cross $t$-intersecting
families in $2^{[n]}$, then one of the following holds.
\begin{enumerate}
 \item $\sqrt{\mu_p(\cA)\mu_p(\cB)} < (1-\gamma\eta)p^{t}$, 
where $\gamma\in(0,1]$ depends only on $t$.
 \item $\mu_p(\cA\bigtriangleup\cF_s^t(n))+
\mu_p(\cB\bigtriangleup\cF_s^t(n))<\eta \mu_p(\cF_s^t(n))$, where $s=0$ or $1$.
\end{enumerate}
If {\rm (ii)} happens then
$\sqrt{\mu_p(\cA)\mu_p(\cB)} < \mu_p(\cF_s^t(n))$ with
equality holding iff $\cA=\cB=\cF_s^t(n)$.
\end{prop}

We need some definitions which we will continue to
use throughout the proof of the main proposition. 

Let $\cF^u$ be the family of all walks
that hit the line $y = x + u$. We identify a subset and 
its walk, so formally, 
\[
\cF^u=\{F\subset[n]: \, |F\cap[j]|\geq(j+u)/2 \text{ for some }j\}.
\]
We partition $\cF^u$ into the following three subfamilies:
\begin{align*}
\tF^u:=&\{F\in\cF^u: \, F \text{ hits $y = x + u+1$}\},\\
\hF^u:=&\{F\in\cF^u: \, F \text{ hits $y = x + u$ exactly once, 
but does not hit $y=x+u+1$}\}, \\
\dhF^u:=
& \{F\in\cF^u: \, F \text{ hits $y = x + u$ at least twice, 
but does not hit $y=x+u+1$}\}. 
\end{align*}
We remark that no walk $F$ in $\hF^u\cup\dhF^u$ hits the line $y=x+(u+1)$. 
This can also be stated as the fact that if $F\in\hF^u\cup\dhF^u$ then
$|F\cap[i]|\leq(i+u)/2$ for all $i$.

If $u+2s\leq n$, then, for simplicity, we also use $\cF_s^u$ 
to mean $\cF_s^u(n)$. This is the family of walks that hit the line 
$y=x+u$ within the first $u+2s$ steps. So if $F\in\cF_s^u$ then there is some
$0\leq i\leq s$ such that the walk corresponding to $F$ hits $(i,u+i)$.

To simplify the notation we write
$ X <_t Y$
if there is a positive function $\gamma=\gamma(t)>0$ depending only on 
$t$ such that $X < (1-\gamma(t))Y$
for all $t\geq 14$ (or $t\geq t_0$ for some specified value 
$t_0\leq 14$). For example, we write
$\mu_p(\cA)\mu_p(\cB)<_t p^{2t}$ 
to mean $\mu_p(\cA)\mu_p(\cB)< (1-\gamma)p^{2t}$ for some $\gamma=\gamma(t)>0$,
which would give (i) of Proposition~\ref{main-prop}.

\subsection{Proof of the Main Proposition: Setup}\label{sec:setup}
Let $t\geq 14$, $0<p\leq\frac 1{t+1}$, $q=1-p$, and $\alpha=p/q$.
Here we record some basic inequalities that will be used frequently:
\begin{equation}\label{basic ineq}
p\leq \frac{1}{t+1},\quad q\geq\frac{t}{t+1}, 
\quad q^{-t}\leq\(1+\frac{1}{t}\)^t< e,
\quad \alpha\leq \frac{1}{t}, \quad pq\leq \frac{t}{(t+1)^2}.
\end{equation}
By Lemma~\ref{monotone}, we may assume that $n$ is sufficiently large.
Let $\cA$ and $\cB$ be shifted, inclusion maximal 
cross $t$-intersecting families in $2^{[n]}$. 

Let $u=\lambda(\cA)$ and $v=\lambda(\cB)$.
Recall that by Lemma~\ref{eq:>=i+t} (ii) we have $u+v\geq 2t$. 
If $u + v \geq 2t+1$, then (i) of Lemma~\ref{lem:hitline} with 
\eqref{basic ineq} yields
 \[ \mu_p(\cA)\mu_p(\cB) \leq 
\alpha^u\alpha^v\leq 
\alpha^{2t+1} = 
p^{2t} \frac{p}{q^{2t+1}} < p^{2t} \frac{e^{2+1/t}}{t+1}.\]
   One can check that $  \frac{e^{2+1/t}}{t+1} <_t 1$ for $t \geq 8$, 
which gives $\mu_p(\cA)\mu_p(\cB)<_t p^{2t}$. 
Thus if $u + v \geq 2t+1$, then (i) of the proposition holds.

From now on let $u+v=2t$. By symmetry we may assume that $u\leq v$. 

We partition $\cA \subset \cF^u$ into families
$\hA:=\cA \cap \hF^u$, $\dhA := \cA \cap \dhF^u$, and $\tA := \cA \cap \tF^u$.  
Similarly, we define
$\hB :=\cB \cap \hF^v$, $\dhB := \cB \cap \dhF^v$, and $\tB := \cB \cap \tF^v$.
We remark that if $A\in\hA\cup\dhA$, then $|A\cap[i]|\leq(i+u)/2$ for all $i$.
Moreover equality holds exactly once if $A\in\hA$, 
and at least twice if $A\in\dhA$.

If $\hA=\emptyset$ then $\cA= \dhA\cup \tA$ and $\muA=\mu_p(\dhA)+\mu_p(\tA)$.
Thus by (i) and (iii) of Lemma~\ref{lem:hitline} we have
\begin{equation}\label{hA=empty}
\mu_p(\cA)\mu_p(\cB)\leq(\alpha^{u+1}+\alpha^{u+1})\alpha^v
\leq 2\alpha^{2t+1}<_t p^{2t},
\end{equation}
where last inequality holds for $t\geq 14$.
(This is the point where we really need $t\geq 14$.)
The same holds for the case when $\hB=\emptyset$.
Thus if $\hA=\emptyset$ or $\hB=\emptyset$ then (i) of the proposition holds.

From now on we assume that $\hA\neq\emptyset$ and $\hB\neq\emptyset$.
Thus there exist $s,s'$ such that $\hA\cap\cF_s^u\neq\emptyset$ and 
$\hB\cap\cF_{s'}^v\neq\emptyset$. 
Remarkably, these $s$ and $s'$ are uniquely determined.
Extending a result in \cite{FF} we show this structural result as follows.

\begin{lemma}\label{lem:structure}
There exist unique nonnegative integers $s$ and $s'$ such that
$s-s'=(v-u)/2$, $\cA_s:=\hA\cup\dhA\subset\cF_s^u$,
and $\cB_{s'}:=\hB\cup\dhB\subset\cF_{s'}^v$.
\end{lemma}

\begin{proof}
Suppose that $\hA\cap\cF_s^u\neq\emptyset$ and 
$\hB\cap\cF_{s'}^v\neq\emptyset$. 
For any $A\in\hA\cap\cF_s^u$ we have
\begin{equation}\label{eq:A cap i}
 |A\cap[i]|\leq (i+u)/2
\end{equation}
for all $i$, with equality holding iff $i=2s+u$.
Similarly, for any $B \in \hB\cap\cF_{s'}^v$, we have
$|B\cap[i]|\leq (i+v)/2$
with equality holding iff $i=2s'+v$. These two inequalities give
\begin{equation}\label{eq:<i+t}
 |A\cap[i]|+ |B\cap[i]|\leq i+(u+v)/2=i+t
\end{equation}
with equality holding iff $i=2s+u=2s'+v$. 
By Lemma~\ref{eq:>=i+t} (i), 
equality must hold in \eqref{eq:<i+t} for this $i$, 
which gives $s-s'=(v-u)/2$. 
If there is some $B'\in\hB\cap\cF_x^v$, then,
by considering $A$ and $B'$, we also have $s-x=(v-u)/2$.
This gives $x=s'$, and hence $\hB\subset\cF_{s'}^v$. 
Similarly we have $\hA\subset\cF_s^u$.

We need to show $\dhA\subset\cF_{s}^u$ and $\dhB\subset\cF_{s'}^v$.
Suppose, to the contrary, that $\dhB\not\subset\cF_{s'}^v$. Then there is
some $B\in\dhB$ such that $|B\cap[i]|\leq(i+v)/2$ for all $i$, and
where equality does {\sl not} hold at $i=2s'+v$. 
For any $A\in\hA\subset\cF_s^u$,
we have \eqref{eq:A cap i} with equality holding only at this same 
$i=2s+u=2s'+v$. So equality in \eqref{eq:<i+t}  never holds for 
these $A$ and $B$, contradicting  Lemma~\ref{eq:>=i+t} (i).
One can show $\dhA\subset\cF_{s}^u$ similarly.
\end{proof}

Here we record our {\bf setup}. 
\begin{itemize}
\item $t\geq 14$, $0<p\leq\frac 1{t+1}$, $q=1-p$, and $\alpha=p/q$.
\item $u+v=2t$, $0\leq u\leq t\leq v\leq 2t$, $s\geq s' \geq 0$, and
$s-s'=(v-u)/2$. 
\item  $u=t-(s-s')$ and $v=t+(s-s')$.
\item $\cA=\hA\cup\dhA\cup\tA\subset\cF^u$, 
$\cB=\hB\cup\dhB\cup\tB\subset\cF^v$,
$\hA\neq\emptyset$ and $\hB\neq\emptyset$.
\item $\cA_s:=\hA\cup\dhA\subset\cF_s^u$ 
and $\cB_{s'}:=\hB\cup\dhB\subset\cF_{s'}^v$.
\end{itemize}

The rest of the proof of the proposition is divided into three parts, 
which break over three more subsections.
In Subsection~\ref{sec:easy cases}, 
we deal with easy cases, namely, all cases but 
$(s,s')=(1,0),(0,0),(1,1)$, and in Subsection~\ref{sec:non-diagonal}
we settle one of the more difficult cases, $(s,s')=(1,0)$. In these cases, 
only (i) of the proposition happens.
Finally in Subsection~\ref{sec:diagonal}, we consider the last two cases 
$(s,s')=(0,0),(1,1)$ where the extremal configurations satisfying (ii) appear.
Then Theorems~\ref{p-thm} and \ref{p-thm-stability} will be easily proved
using the proposition.

\subsection{Proof of Proposition \ref{main-prop}: Easy cases}\label{sec:easy cases}
Let $\bF^r_i:=(\hF^r\cup\dhF^r)\cap\cF^r_i$.

\begin{claim}\label{clm:f} Let $r\geq 1$ and $i\geq 0$ be integers, let $p\leq p_0<1/2$, and let
$\epsilon=0.001$. There exists an $n_0$ such that  for all $n\geq n_0$  we have
$$
\mu_p(\tF^r\cup \bF^r_i) < p^r f(r,i,p)(1+\epsilon),
$$
where
$$
f(r,i,p)=\frac{p}{q^{r+1}}+ \binom{2i+r}{i}\frac{r+1}{r+i+1}
\(1-\frac{p}{q}\)(pq)^i.
$$
\end{claim}

\begin{proof}
Let $L$ be the line $y=x+r+1$. Then every walk in $\tF^r$ hits $L$, and
(i) of Lemma~\ref{lem:hitline} yields
\begin{equation}\label{eq:aux_1}
\mu_p(\tF^r)\leq \alpha^{r+1}=\frac{p^{r+1}}{q^{r+1}}\leq \frac{p^{r+1}}{q^{r+1}}(1+\epsilon).\end{equation}

A walk $W \in \bF^r_i$ must go from $O=(0,0)$ to $Q = (i,i+r)$ without hitting 
the line $L$, and then continue on without hitting $L$. 
It follows from (ii) of Lemma~\ref{lem:k-hitline} that
the number of walks from $O$ to $Q$ that do not hit $L$ is
\[
\binom{i+(i+r)}{i}-\binom{i+(i+r)}{(i+r)-(r+1)}=
\binom{2i+r}{i} \frac{r+1}{r+i+1},
\]
so the probability of a random walk 
$W_{2i+r,p}$ hitting $Q$ without hitting $L$ is 
$\binom{2i+r}{i} \frac{r+1}{r+i+1}p^{r+1}q^i$.
By Lemma~\ref{lem:hitline} (ii), the random walk continues on from $Q$ 
without hitting $L$, with probability at most 
$1-\alpha+\delta\leq(1-\alpha)(1+\epsilon)$, where
$\delta=\frac{1-2p_0}{1-p_0}\epsilon$. Therefore we have
\begin{equation}\label{eq:aux_2}
\mu_p(\bF^r_i)
\leq \binom{2i+r}i\frac{r+1}{r + i +1}p^{r+i}q^i(1-\alpha)(1+\epsilon).
\end{equation}
Combining~\eqref{eq:aux_1} and~\eqref{eq:aux_2} completes the proof of Claim~\ref{clm:f}.
\end{proof}

\begin{lemma}\label{easycases}
If $s\geq 2$, then $\muA\muB<_t p^{2t}$.
\end{lemma}

\begin{proof}
Since $\muA\muB\leq\mu_p(\tF^u\cup \bF^u_s)\mu_p(\tF^v\cup \bF^v_{s'})$,
it suffices to show the RHS is $<_t p^{2t}$.
We will show that $f(u,s,p)f(v,s',p)<_t 0.99$ for all $n\geq n_0$. We claim that it implies Lemma~\ref{easycases}.
Indeed, since $0.99(1+\epsilon)^2< 0.992$ and Claim~\ref{clm:f} holds,
we have $\muA\muB<_t p^{2t}$ for all $n\geq n_0$.
By Lemma~\ref{monotone}, this implies Lemma~\ref{easycases} for all $n$.

Now we show that $f(u,s,p)f(v,s',p)<_t 0.99$ for all $n\geq n_0$.
Since $u\leq t$ and $f(u,s,p)$ is an increasing function of $u$, we have
$f(u,s,p)\leq f(t,s,p)$. The first term of $f(t,s,p)$ is clearly
increasing in $p$, and the second term is also an increasing
function of $p$ iff $\frac 1p+4p>4+\frac 1s$, which is certainly true
for $p\leq\frac 1{t+1}\leq 1/15$ and $s\geq 1$.
Thus we have $f(t,s,p)\leq f(t,s,\frac1{t+1})=:g(s,t)$.
By a direct computation we see that $g(s,t)>g(s+1,t)$ iff
$$
\frac{(t+1)^2(s+1)(s+t+2)}{t(2s+t+2)(2s+t+1)}>1,
$$
or equivalently,
$$s^2(t-1)^2+s(t^3+t^2+t+3)+(t^2+3t+2)>0$$
which is true for $t\geq 1$ and $s\geq 0$.
Similarly, noting that $v\leq 2t$, we have
$$\ts
f(v,s',p)\leq f(2t,s',p)\leq f(2t,s',\frac1{t+1}):=h(s',t),
$$
and $h(s',t)>h(s'+1,t)$ holds for all $t\geq 1$ and $s'\geq 1$.
Thus for $t\geq 14$, $s\geq 3$ and $s'\geq 1$ we have
$$\ts
f(u,s,p)f(v,s',p)\leq g(s,t)h(s',t)\leq g(3,t)h(1,t)\leq
g(3,14)h(1,14)<0.87.
$$

The remaining cases are $s'=0$ or $s=2$. If $s'=0$ and $s\geq 2$ 
then, observing that $f(v,0,p)$ is (decreasing in $p$ but) bounded by $1$ as 
$p$ goes to $0$,  we have
$$f(u,s,p)f(v,s',p)\leq g(2,14)\cdot 1 < 0.96.$$
If $s=2$ and $s'=1$ then $u=t-1$, $v=t+1$, and
$$\ts
f(u,s,p)f(v,s',p)\leq f(t-1,2,\frac1{t+1})f(t+1,1,\frac1{t+1})
\leq f(13,2,\frac1{15})f(15,1,\frac1{15})<0.68.$$
If $s=s'=2$ then $u=v=t$ and
$$\ts
f(u,s,p)f(v,s',p)\leq f(t,2,\frac1{t+1})^2\leq f(14,2,\frac1{15})^2<0.46.$$
This completes the proof of Lemma~\ref{easycases}.
\end{proof}

We have proved the proposition for the cases $s\geq 2$, and
the remaining cases are $s\leq 1$, namely, $(s,s')=(0,0),(1,1),(1,0)$.
We will discuss these three cases in the next two subsections.
 
\subsection{Proof of Proposition \ref{main-prop}: A harder case}\label{sec:non-diagonal}
Here we consider the case $(s,s')=(1,0)$. In this case, we have $u=t-1$ and $v=t+1$. 
\begin{lemma}\label{lem:(s,s')=(1,0)}
 We have $\mu_p(\cA)\mu_p(\cB)<_t p^{2t}$ for $(s,s')=(1,0)$.
\end{lemma}
\begin{proof}
For $1\leq i\leq n-t-2$, let
$$D^{\cA}_i=[t-2]\cup \{t, t+1\}\cup \{t+1+i+2\ell\in [n] : \ell=1,2,\dots\}\in 
\hF^{t-1}\cap\cF_1^{t-1},$$
and for $1\leq j\leq n-t-2$, let
$$D^{\cB}_j= [t+1]\cup \{t+1+j+2\ell \in [n] : \ell=1,2,\dots\} \in 
\hF^{t+1}\cap\cF_0^{t+1}.$$
(See Figure~\ref{fig:D_I_J}.)

\begin{figure}
\includegraphics[scale=0.18]{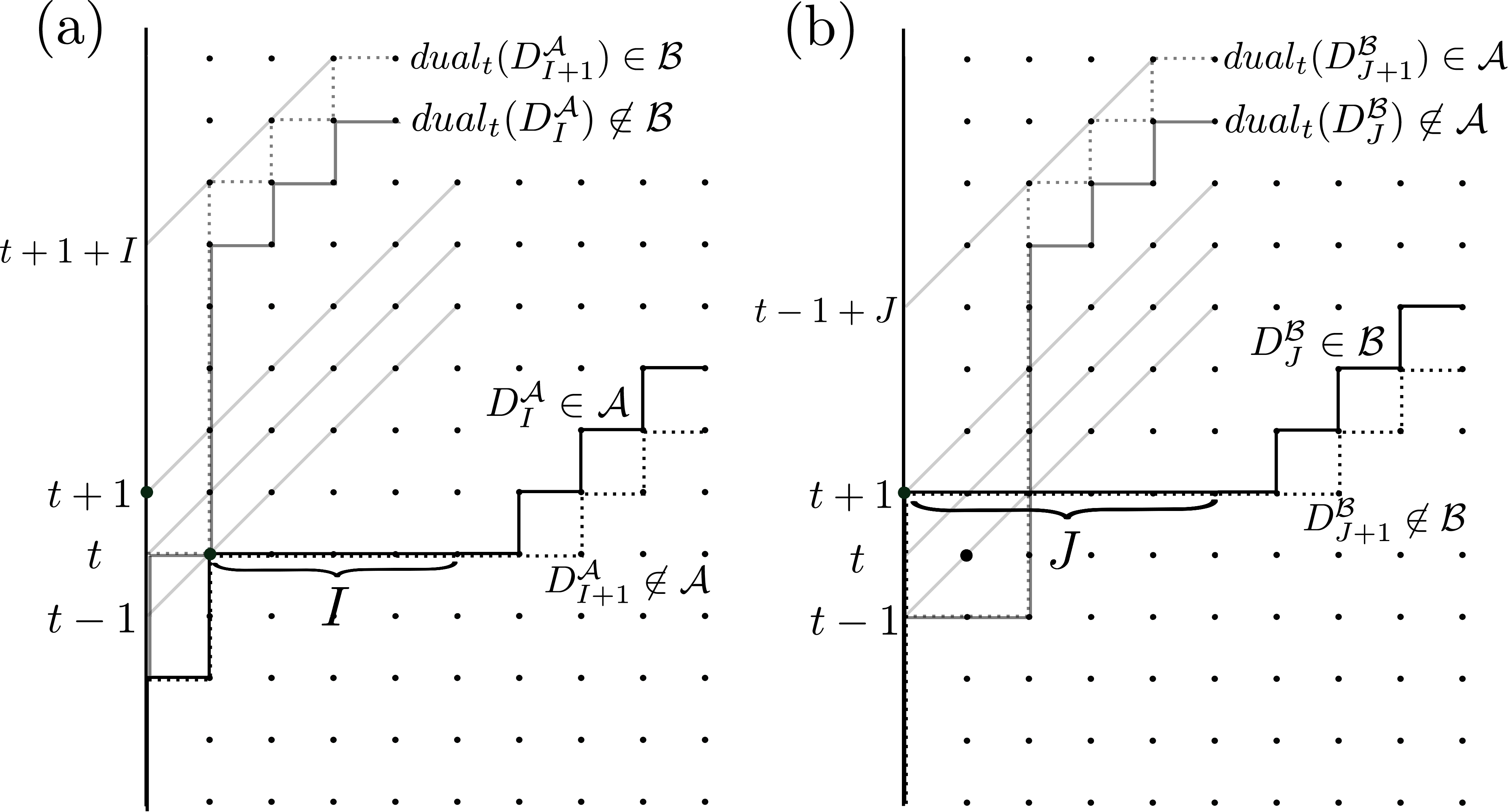}
\caption{(a) $D^{\cA}_I$ and (b) $D^{\cB}_J$}
\label{fig:D_I_J}
\end{figure}

Let $I:=\max\{ i : D^{\cA}_i\in \cA\}$.
We claim that $I$ is well defined. Indeed, we have assumed that 
$\emptyset\neq\hA\subset\cA_1$ (see setup in Subsection~\ref{sec:setup}), 
and $A\shiftsto D^{\cA}_1$ for any $A\in\hA$.
Thus we have that $D^{\cA}_1\in \cA$, and so
$\{ i : D^{\cA}_i\in \cA\}\neq\emptyset$.
Similarly, we have
$\emptyset\neq\hB\subset\cB_0$ and $B\shiftsto D^{\cB}_1$ for any $B\in\hB$.
Thus $D^{\cB}_1\in \cB$ and we can define $J:=\max\{j:D^{\cB}_j\in\cB\}$.

 We consider several cases separately. During this subsection, $\epsilon$ is an arbitrarily small constant,
 depending only on $n$, where $\epsilon\rightarrow 0$ as $n\rightarrow \infty$.

\smallskip\noindent{\bf Case 1.} When $I\geq 2$ and $J\geq 2$.

First we show 
\begin{equation}\label{clm:J>1}
\muA <  p^{t}\(e\alpha+1+tq+t\epsilon\).
\end{equation}
Note that $\mu_p(\cA)=\mu_p(\tA)+\mu_p(\cA_1)$. 
To estimate $\mu_p(\tA)$, let $A\in\tA$. Then $A$ hits the line
$L_0:y=x+t$. At the same time,
since the dual walk $\dual_t(D^{\cB}_J)$ is not in $\cA$, where
\[
 \dual_t(D^{\cB}_J)=[t-1]\cup[t+2,t+J+2]\cup\{t+J+4,t+J+6,\ldots,\},
\]
we have that $A\not\to\dual_t(D_J^{\cB})$ (see Figure~\ref{fig:D_I_J}).
To satisfy these conditions, it suffices for $A$ that one of the following
properties holds:

\begin{itemize}
\item[(1a)] $A$ hits the line $L_1:y=x+(t+J-1)$. (Then $A$ hits the line $L_0$
automatically.)
\item[(1b)] $A$ does not hit $L_1$, but hits $(0,t)$. 
(Notice that $(0,t)$ is on $L_0$, and a walk hitting this point cannot 
shift to $\dual_t(D^{\cB}_J)$.)
\item[(1c)] $A$ does not hit $L_1$ or $(0,t)$, 
but hits $(1,t)$ and the line $L_0$.
(Notice that a walk hitting $(1,t)$ cannot 
shift to $\dual_t(D^{\cB}_J)$.)
\end{itemize}
Thus we have
\begin{align*} 
\mu_p(\tA) &\leq\mu_p(\mbox{walks of (1a)})+
\mu_p(\mbox{walks of (1b)})+\mu_p(\mbox{walks of (1c)})
\nonumber\\
&\leq  \alpha^{t+J-1}+p^t(1-\alpha^{J-1}+\epsilon) +tp^tq(\alpha-\alpha^J+\epsilon).
\end{align*}
The last inequality uses Lemma~\ref{lem:hitline}~(i) for the first term. For the second and third terms, we 
use Lemma~\ref{lem:hitline}~(ii) in combination with Example~\ref{ex:hitpoint}.
Hence,
\begin{equation}\label{eq:tA}
\mu_p(\tA) \leq  \alpha^{t+J-1}+p^t +tp^tq\alpha.
\end{equation}

To bound $\mu_p(\cA_1)$ we simply use $\cA_1\subset\bF_1^{t-1}$ and
simply bound $\mu_p(\bF_1^{t-1})$. 
Let $F\in\bF_1^{t-1}=\cF_1^{t-1}\cap(\hF^{t-1}\cup \dhF^{t-1})$. 
Then, it follows from $F\in \cF_1^{t-1}$ that
\[
|F\cap[(t-1)+2]|\geq(t-1)+1,
\] 
that is, $F$ hits $(0,t+1)$ or $(1,t)$.
On the other hand, it follows from $F\in\hF^{t-1}\cup\dhF^{t-1}$ that $F$ hits the
line $y=x+t-1$, but does not hit $y=x+t$. Combining these things, we have 
that $F\in\bF_1^{t-1}$ hits $(1,t)$ without hitting $(0,t)$,
and then from $(1,t)$ it will never hit $y=x+1$. Therefore we have
\begin{equation}\label{eq:hA}
\mu_p(\cA_1)\leq \mu_p(\bF_1^{t-1}) \leq tp^tq(1-\alpha+\epsilon)=tp^tq(1+\epsilon)-tp^tq\alpha,
\end{equation}
where again we use Lemma~\ref{lem:hitline}~(ii) in combination with Example~\ref{ex:hitpoint}.

Combining~\eqref{eq:tA} and~\eqref{eq:hA} implies that
\begin{align*}
\muA&=\mu_p(\tA)+\mu_p(\cA_1)
\leq  \alpha^{t+J-1}+p^t+tp^tq(1+\epsilon) \\
&= p^{t}\(\frac{1}{q^{t}}\alpha^{J-1}+1+tq(1+\epsilon)\) 
< p^{t}\(e\alpha+1+tq+t\epsilon\),
\end{align*}
where the last inequality follows from $q^{-t}\leq(1+\frac 1t)^t<e$ 
and  $J\geq 2$. This proves \eqref{clm:J>1}.

Next we show
\begin{equation}\label{clm:I}
\muB   \leq  p^{t}\(e\alpha^3+p+p\epsilon\).
\end{equation}
Since $\mu_p(\cB)=\mu_p(\tB)+\mu_p(\cB_0)$ we estimate the $p$-weights of
$\tB$ and $\cB_0$ separately.

Let $B\in\tB$. It follows from $\tB\subset\bF^{t+1}$ that $B$ hits the line
$y=x+t+2$. On the other hand, it follows from $\dual_t(D^{\cA}_I)\not\in \cB$
that $B$ hits $(0,t+1)$ or the line $y=x+(t+1+I)$. Thus we get
\begin{align}\label{eq:tB}
\mu_p(\tB)& < \mu_p(\mbox{walks in $\tB$ hitting $y=x+(t+1+I)$})\nonumber \\ 
&\quad+\mu_p(\mbox{walks in $\tB$ hitting both $(0,t+1)$ and $y=x+t+2$})\nonumber \\
&\leq  \alpha^{t+1+I}+p^{t+1}\alpha.
\end{align} 

As for $\cB_0\subset\bF^{t+1}_0$, 
noting that walks in $\bF^{t+1}_0$ hit $(0,t+1)$ but do not
hit the line $y=x+t+2$, we obtain
\begin{equation}\label{eq:hB}
\mu_p(\cB_0) 
\leq  \mu_p(\bF^{t+1}_0) \leq  p^{t+1}(1-\alpha+\epsilon)=p^{t+1}(1+\epsilon)-p^{t+1}\alpha.
\end{equation}
Combining~\eqref{eq:tB} and~\eqref{eq:hB} yields that
\begin{align*}
\muB &= \mu_p(\tB)+\mu_p(\cB_0) 
 \leq  \alpha^{t+1+I}  +  p^{t+1}\alpha +p^{t+1}(1+\epsilon)-p^{t+1}\alpha \nonumber \\
 &=  p^{t}\(\frac{1}{q^{t}}\alpha^{1+I}+p(1+\epsilon)\)
  <  p^{t}\(e\alpha^3+p+p\epsilon\),
\end{align*}
where the last inequality follows from $q^{-t}<e$ and  $I\geq 2$.
This proves \eqref{clm:I}.

Now we are ready to show $\muA\muB<_t p^{2t}$.
By \eqref{clm:J>1} and \eqref{clm:I} we have
\begin{align}\label{case1:g}
\muA\muB &<  p^{2t}\(e\alpha+1+tq+t\epsilon\)\(e\alpha^3+p+p\epsilon\)
\nonumber \\
&=p^{2t}(ep\alpha+p+tpq+tp\epsilon)(ep^2/q^3+1+\epsilon)\nonumber \\
&< p^{2t}\(\frac{e}{t(t+1)}+\frac{1}{t+1}+\frac{t^2}{(t+1)^2}+\epsilon\)
\(e\frac{t+1}{t^3}+1+\epsilon\) \nonumber \\
&< p^{2t}\(\(\frac{e}{t(t+1)}+\frac{1}{t+1}+\frac{t^2}{(t+1)^2}\)
\(e\frac{t+1}{t^3}+1\)+4\epsilon\) \nonumber \\
&=: p^{2t} \(g(t)+4\epsilon\), 
\end{align}
where the second inequality follows from \eqref{basic ineq}.
Note that $\frac{d}{dt}g(t)<0$ for $t\leq 13$ while $\frac{d}{dt}g(t)>0$ 
for $t\geq 14$. In addition, $g(7)<0.999$, and  
$\lim_{t\rightarrow \infty}g(t)=1$. Hence, we have $g(t)<_t 1$ for all 
$t\geq 7$. Hence we have $\muA\muB<_t p^{2t}$ for all $t\geq 7$.
This completes the proof for Case 1.

\smallskip\noindent{\bf Case 2.} When $I=1$.

We first estimate $\muA$.
Trivially, 
\[
\mu_p(\tA)\leq \alpha^t=\frac{p^t}{q^t}\leq p^t\(1+\frac1t\)^t. 
\]
Next we consider the walks in $\cA_1$. These walks hit the line $y=x+t-1$ 
but do not hit $y=x+t$, and $\cA_1\subset\cF_1^{t-1}$ implies that they 
hit $(0,t+1)$ or $(1,t)$.
Consequently, walks in $\cA_1$ hit $(1,t)$ without
hitting the $y=x+t$. The weight of these walks is at most
$tp^t q (1-\alpha+\epsilon)$.
Among them, we look at the 
walks that hit all of $(1,t-2)$, $(1,t)$, $(3,t)$, and 
do not hit $y=x+t-2$ after hitting $(3,t)$. These walks cannot be in $\cA_1$, as they
shift to 
\[
 D^{\cA}_2=[t-2]\cup\{t,t+1\}\cup\{t+5,t+7,\ldots\};
\]
but $D^{\cA}_2\not\in \cA$. 
The weight of such walks is at least $(t-1)p^tq\cdot q^2(1-\alpha)$.
Hence we infer 
\begin{align*}
\mu_p(\cA_1) &\leq tp^t q (1-\alpha+\epsilon)
-(t-1)p^tq\cdot q^2(1-\alpha)\\
&=p^t\((1-\alpha)(tq-(t-1)q^3)+tq\epsilon\).
\end{align*}
Then we use the fact that $(1-\alpha)(tq-(t-1)q^3)$ is increasing in $p$
for $0\leq p\leq \frac 1{t+1}$, which gives
\begin{equation}\label{eq:3t+1}
 (1-\alpha)(tq-(t-1)q^3)\leq\frac{t(t-1)(3t+1)}{(t+1)^3}.
\end{equation}
Thus we have 
\begin{align}\label{eq:A_I=1}
\muA &\leq \mu_p(\tA)+\mu_p(\cA_1) 
\leq p^{t}\(\(1+\frac1t\)^t+\frac{t(t-1)(3t+1)}{(t+1)^3}+t\epsilon\).
\end{align}

On the other hand, we trivially have 
\[
\muB\leq\alpha^{t+1}=\frac{p^{t+1}}{q^{t+1}}
\leq p^t\(1+\frac1t\)^{t+1}\frac1{t+1}=p^t\(1+\frac1t\)^{t}\frac1{t}.
\]
Combining this with~\eqref{eq:A_I=1} yields that
\begin{align*}
\frac{\muA\muB}{p^{2t}} &\leq 
g(t)+t\epsilon\(1+\frac1t\)^{t}\frac1{t}<g(t)+e\epsilon,
\end{align*} 
where 
\begin{equation}\label{def:g(t)}
 g(t):=\(\(1+\frac1t\)^t+\frac{t(t-1)(3t+1)}{(t+1)^3}\)
\(1+\frac1t\)^{t}\frac1{t}.
\end{equation}
Thus it suffices to show that $g(t)<_t1$. By direct computation we see that
$g(13)<1$. For $t\geq 14$ we use $(1+\frac 1t)^t<e$ again to get
\begin{equation}\label{def:g2(t)}
 g(t)<\(\frac et+\frac{(t-1)(3t+1)}{(t+1)^3}\)e=:g_2(t). 
\end{equation}
The RHS is decreasing in $t$ for $t > 0$, and is less than $1$ when
$t=14$. This completes the proof for Case 2.

\smallskip\noindent{\bf Case 3.} When $J=1$.

Clearly, we have
\[
\muA\leq \alpha^{t-1}=p^{t-1}\frac{1}{q^{t-1}}, 
\]
and 
\[
\mu_p(\tB)\leq \alpha^{t+2}=\frac{p^{t+2}}{q^{t+2}}. 
\]

As for $\cB_0$ we count the walks that hit $(0,t+1)$ and do not hit
the line $y=x+t+2$. Among them we delete the walks that hit both $(0,t+1)$
and $(2,t+1)$, and do not hit the line $y=x+t-1$ after hitting $(2,t+1)$.
(If such walk was in $\cB$, then this would give $D_2^{\cB}\in\cB$, which
is a contradiction.) Thus we have
\begin{align*}
\mu_p(\cB_0)&\leq p^{t+1}(1-\alpha+\epsilon-q^2(1-\alpha)). 
\end{align*}
Hence we infer
\begin{align*}
\muA\muB &\leq  p^{2t}\frac{1}{q^{t-1}}\(\frac{p}{q^{t+2}}+1-\alpha+\epsilon-q^2(1-\alpha)\) \\
&\leq p^{2t}e\(e\frac{(t+1)^2}{t^2}p+(1-\alpha)(1-q^2)+\epsilon\) \\
&\leq p^{2t}\(e^2\frac{(t+1)}{t^2}+e\frac{(t-1)(2t+1)}{t(t+1)^2}+e\epsilon\)
=: p^{2t}\(h(t)+e\epsilon\),
\end{align*} 
where the second inequality follows from $\frac{1}{q}\leq \frac{t+1}{t}$ and $\frac{1}{q^t}\leq (1+\frac{1}{t})^t<e$, and the third inequality follows from $p\leq \frac{1}{t+1}$ 
and the fact that the function $(1-\alpha)(1-q^2)$ is increasing in $p$  for $p\leq 0.274$.
Since $\frac{dh(t)}{dt}<0$ and $h(13)<0.96$, we have that $h(t)<0.96$ for all $t\geq 13$, and hence, for all $t\geq 13$,
$$
\muA\muB  < 0.97 p^{2t}.
$$
This completes the proof for Case 3, and so for  
Lemma~\ref{lem:(s,s')=(1,0)}.
\end{proof}

We state now a partial version of Proposition~\ref{main-prop}, 
recording what we have proved so far.  

\begin{cor}\label{weak-prop}
For every $t\geq 14$, $n\geq t$, and $p$ with $0<p\leq\frac 1{t+1}$,
we have the following. If $\cA$ and $\cB$ are shifted,
inclusion maximal cross $t$-intersecting
families in $2^{[n]}$, then one of the following holds.
\begin{enumerate}
 \item $\sqrt{\mu_p(\cA)\mu_p(\cB)} < (1-\gamma)p^{t}$, 
where $\gamma\in(0,1]$ depends only on $t$.
 \item $(s,s')=(0,0),(1,1)$.
\end{enumerate}
\end{cor}

\subsection{ Proof of Proposition \ref{main-prop}: extremal cases}\label{sec:diagonal}

Finally, we consider the cases $(s,s')=(0,0), (1,1)$. In these cases $u=v=t$.  

Recall that
$\cF_s^t=\{F\subset[n]:|F\cap[t+2s]\geq t+s\}$.
Let
$$
D_i=[1,t-1]\cup\{t+s, t+2s\}\cup\{t+2s+i+2j\in [n]:j=1,2,\ldots\}\in
\hF^t\cap\cF_s^t
$$
for $1\leq i\leq n-t-2s-1=:i_{\max}$. (See Figure~\ref{fig:D_I}.)
\begin{figure}
\includegraphics[scale=0.18]{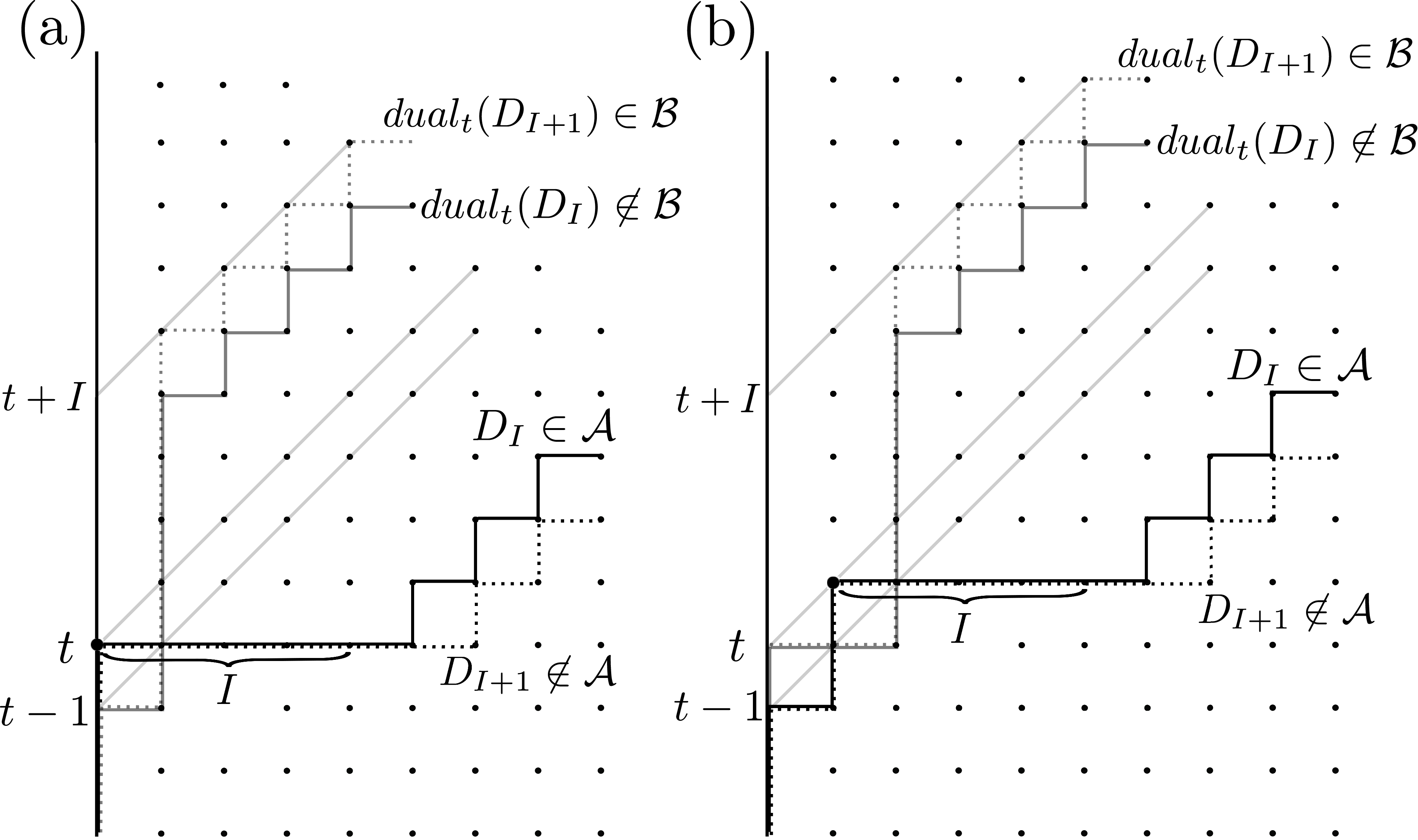}
\caption{The walk $D_I$ when (a) $s = 0$ and (b) $s = 1$}\label{fig:D_I}
\end{figure}
Notice that $$D_{i_{\max}}=[1,t-1]\cup\{t+s, t+2s\}$$ and
$$D_{i_{\max}-1}=[1,t-1]\cup\{t+s, t+2s\}\cup\{n\}.$$
Let $I:=\max\{ i : D_i\in \cA\}$. We claim that $I$ exists.
Indeed, one can check that $A \shiftsto D_1$ for any $A \in \hA \cap \cA_s$, 
and $D_1 \in \cA$.
(Recall from the setup at the end of Subsection~\ref{sec:setup} 
that we assume $\emptyset\neq\hA\subset\cA_s$.)  
Similarly, $J:=\max\{ j : D_j\in \cB\}\neq\emptyset$.

\begin{claim}\label{I neq imax}
If $I\neq i_{\max}$, then 
$\mu_p(\cB\setminus\cF^t_s)<_t \mu_p(\cF^t_s\setminus\cA)$.
\end{claim}

\begin{proof}
We first show that
\begin{equation}\label{eq:F-A} \mu_p(\cF^t_s\setminus \cA)\geq {t \choose s}p^{t+s}q^{s+I+1}(1-\alpha)
\end{equation}
and
\begin{equation}\label{eq:B-F}
\mu_p(\cB\setminus \cF^t_s)\leq \alpha^{t+I}.
\end{equation}

Consider a walk $W$ that hits $(s,t+s)$ and shifts to $D_{I+1}$. 
Since $D_{I+1}\not\in \cA$ we have $W\in\cF^t_s \setminus \cA$.
Further, such a walk $W$ must hit $Q_1=(s,t-s)$ and $Q_2=(s+I+1,t+s)$. There are
$\binom ts$ ways for $W$ to go from $(0,0)$ to $Q_1$, then the next $2s+I+1$
steps to $Q_2$ are unique. A random walk $W_{t+2s+I+1,p}$ 
has this property
with probability $\binom{t}{s}p^{t+s}q^{s+I+1}$.     
From $Q_2$, a point on the line $y = x + (t - I - 1)$, the walk
must not hit $y = x + (t - I)$. (Otherwise $W\shiftsto D_{I+1}$ fails.) 
This happens, by Lemma \ref{lem:hitline}~(i), with probability at least
$1 - \alpha$, which gives \eqref{eq:F-A}.

Next we show~\eqref{eq:B-F}. Since $\dual_t(D_I)\not\in\cB$, 
each walk in $\cB$ hits at least one of $(0,t+s)$, $(s,t+s)$, and $y=x+(t+I)$. 
Since each walk hitting $(0,t+s)$ or $(s,t+s)$ is in $\cF^t_s$, each walk in $\cB\setminus \cF^t_s$ hits $y=x+(t+I)$. This yields~\eqref{eq:B-F}.

Therefore it suffices, by \eqref{eq:F-A} and \eqref{eq:B-F}, to show 
$\alpha^{t+I}<_t {t \choose s}p^{t+s}q^{s+I+1}(1-\alpha)$.
We have
$$
\frac{{t \choose s}p^{t+s}q^{s+I+1}(1-\alpha)}{\alpha^{t+I}}={t \choose s}p^sq^{t+s+1}(1-\alpha)\(\frac{q^2}{p}\)^I 
\geq  {t \choose s}p^{s-1}q^{t+s+2}(q-p),
$$
where the first inequality holds because of $q^2/p>1$ for $p<0.38$. 
Since $p\leq 1/(t+1)$, one can easily check that ${t \choose s}p^{s-1}q^{t+s+2}(q-p)>_t 1$
 if $s=0$ and $t\geq 5$, or if $s=1$ and $t\geq 6$. 
\end{proof}

The following part will also be used in proving $k$-uniform results.
To make this reuse easier we introduce some names as follows. Let
$f=\mu_p(\cF_s^t)$,
$a=\mu_p(\cA)$,
$a_0=\mu_p(\cA\cap\cF_s^t)$,
$a_1=\mu_p(\cA\bigtriangleup\cF_s^t)$,
$a_f=\mu_p(\cA\setminus\cF_s^t)$, and
$f_a=\mu_p(\cF_s^t\setminus\cA)$.
(So $f=a_0+f_a$, $a=a_0+a_f$, and $a_1=a_f+f_a$.)
Define $b,b_0,b_1,b_f, f_b$ similarly.

\begin{lemma}\label{lemma:not-extremal}
Let $\eta>0$ be given.
If $I\neq i_{\max}$, then one of the following holds.
\begin{enumerate}
\item $\sqrt{ab}<(1-\frac{\beta\eta}4)f$, where $\beta\in(0,1]$ depends only on $t$.
\item $a_1+b_1<\eta f$ and $\sqrt{ab}<f$.
\end{enumerate}
\end{lemma}

\begin{proof}
We first show that there exists $\beta=\beta(t)>0$ such that
\begin{equation}\label{eq:beta1}
 b_f\leq(1-\beta)f_a
\text{ and }
 a_f\leq(1-\beta)f_b.
\end{equation}
By Claim~\ref{I neq imax}, there is $\beta_1=\beta_1(t)>0$ such that 
$b_f\leq(1-\beta_1)f_a$.
Similarly if $J\neq i_{\max}$, then there is $\beta_2>0$ such that 
$a_f\leq(1-\beta_2)f_b$. 
If $J=i_{\max}$, then $\cA\subset\cF_s^t$, that is, $a_f=0$, and
$a_f\leq(1-\beta_2)f_b$ holds for any $\beta_2<1$. 
Thus, letting $\beta=\min\{\beta_1,\beta_2\}=\beta(t)>0$, we have
\eqref{eq:beta1}.

Now suppose that $a_0+b_0<\(1-\frac{\eta}4\)2f$.
Then, using \eqref{eq:beta1}, we have
\begin{align}
a+b&=a_0+a_f+b_0+b_f\label{A+B}
\leq(1-\beta)
\(a_0+f_a+b_0+f_b\) + \beta\(a_0+b_0\) \nonumber\\
&\leq (1-\beta)2f+\beta\(a_0+b_0\) 
< (1-\beta)2f+\beta\(1-\frac{\eta}4\)2f
=\(1-\frac{\beta\eta}4\) 2f. \nonumber
\end{align}
Thus we have
$\sqrt{ab}\leq\frac{a+b}2<
(1-\frac{\beta\eta}4) f$.

Next suppose that 
$a_0+b_0\geq\(1-\frac{\eta}4\)2f$.
This gives 
$f_a+f_b\leq\frac{\eta}2f$.
Thus, using \eqref{eq:beta1}, we have
$a_1+b_1=a_f+f_a+b_f+f_b<2\(f_a+f_b\)\leq\eta f$.
Also it follows from \eqref{eq:beta1} that
$a+b=a_0+b_0+a_f+b_f<a_0+b_0+f_a+f_b\leq 2f$
which gives $\sqrt{ab}<f$.
\end{proof}

If $I\neq i_{\max}$, then one of (i) or (ii) of Proposition~\ref{main-prop}
holds by Lemma~\ref{lemma:not-extremal}.
(In this case we always have $\sqrt{ab}<f$.)
The same holds for the case $J\neq i_{\max}$.

Consequently we may assume that $I=J=i_{\max}$.
It follows from $I=i_{\max}$ that $D_{i_{\max}}\in \cA$, and hence the dual,
$\dual_t(D_{i_{\max}})=[n]\setminus\{t+s,t+2s\}$ is not in $\cB$. 
Thus all walks $B$ in $\cB$ satisfy $ B \not\shiftsto \dual_t(D_{i_{\max}})$, 
and $\cB\subset\cF_s^t$ holds.
Also, $J=i_{\max}$ yields $\cA\subset\cF_s^t$. 
In this situation, we clearly have $\sqrt{ab}\leq f$
with equality holding iff
$\cA=\cB=\cF_s^t$.
Thus all we need to do is to show that one of (i) or (ii) of 
Proposition~\ref{main-prop} holds.
Let $f_a=\xi_af$, $f_b=\xi_bf$, and let $\xi=\xi_a+\xi_b$.
Then $a_1+b_1=f_a+f_b=\xi f$. On the other hand it follows that
$\sqrt{ab}=\sqrt{a_0b_0}=\sqrt{(1-\xi_a)(1-\xi_b)}f
\leq\frac{(1-\xi_1)+(1-\xi_b)}2f=(1-\frac{\xi}2)f\leq(1-\frac{\xi}2)p^t$.
Now let $\eta$ be given. If $\xi<\eta$, then (ii) holds.
If $\xi\geq\eta$, then (i) holds by taking $\gamma$ slightly smaller than $1/2$.
This completes the whole proof of Proposition~\ref{main-prop}. \qed

\begin{proof}[Proof of Theorem~\ref{p-thm}]
This follows from Proposition~\ref{main-prop} if $\cA$ and $\cB$ are shifted.
(Recall that if $0<p\leq\frac1{t+1}$ then 
$\mu_p(\cF_0^t)\geq\mu_p(\cF_1^t)$ with equality holding iff $p=\frac1{t+1}$.)
If they are not shifted, then we use Lemma~\ref{lem:shifting} (iii) to
get shifted families $\cA'$ and $\cB'$ starting from $\cA$ and $\cB$.
By Lemma~\ref{lem:shifting} (i) they have the same 
$p$-weights as $\cA$ and $\cB$,
and so by Proposition~\ref{main-prop}, 
\[
\sqrt{\mu_p(\cA')\mu_p(\cB')}=\sqrt{\mu_p(\cA)\mu_p(\cB)}\leq
\mu_p(\cF_s^t)\leq p^{t} \quad (s=0,1).  
\]
Moreover if both equalities hold then either
$\cA'=\cB'=\cF_0^t$, or $p=\frac1{t+1}$ and $\cA'=\cB'=\cF_1^t$,
and Lemma~\ref{lem:shifting} (iv) (or Lemma~\ref{non-unif-shifting})
gives us that either
$\cA = \cB \cong \cF_0^t$, or $p = \frac{1}{t+1}$ and $\cA=\cB\cong\cF_1^t$.
\end{proof}

\begin{proof}[Proof of Theorem~\ref{p-thm-stability}]
This directly follows from Proposition~\ref{main-prop} unless
(ii) of Proposition~\ref{main-prop} happens with $s=1$.
In this last case, we notice that
\[
g(p):=\mu_p(\cF_1^t)/\mu_p(\cF_0^t)=(t+2)p(1-p)+p^2
\]
is an increasing function of $p$ on $(0,\frac1{t+1}(1+\frac t2)]$, 
and $g(\frac1{t+1})=1$. Thus we have 
\[
\sqrt{\mu_p(\cA)\mu_p(\cB)}\leq\mu_p(\cF_1^t(n))<
\ts g(\frac1{t+1}-\epsilon)p^t. 
\]
This gives (i) of Theorem~\ref{p-thm-stability} 
by choosing $\gamma$ so that $g(\frac 1{t+1}-\epsilon)=1-\gamma\eta$.
\end{proof}

\section{Results about uniform families}\label{sec:k-uni}

In this section, we prove Proposition \ref{main-k-prop} 
about $k$-uniform cross $t$-intersecting families, from which
Theorems \ref{k-thm} and \ref{k-thm-stability} will follow.


\begin{prop}\label{main-k-prop} 
For every $k \geq t \geq 14$, $n \geq (t+1)k$  and $\eta\in(0,1]$
we have the following. If $\cA$ and $\cB$ are shifted
cross $t$-intersecting families in $\binom{[n]}k$, 
then one of the following holds.
\begin{enumerate}
 \item $\sqrt{|\cA||\cB|} < (1-\gamma^*\eta)\binom{n-t}{k-t}$, 
where $\gamma^*\in(0,1]$ depends only on $t$.
 \item $|\cA\bigtriangleup\cF_s^t(n,k)|
+|\cB\bigtriangleup\cF_s^t(n,k)|<\eta|\cF_s^t(n,k)|$, where $s=0$ or $1$.
\end{enumerate}
If (ii) happens then $\sqrt{|\cA||\cB|}\leq |\cF_s^t(n,k)|$ with
equality holding iff $\cA=\cB=\cF_s^t(n,k)$.
\end{prop}

Our proof of Proposition \ref{main-k-prop} closely follows the proof of Proposition \ref{main-prop}.
We will use $k$-uniform versions of the concepts of that proof, but instead of introducing 
another index $k$, we redefine our notation. In particular we let 
$\cF^u$ be the family of walks from $(0,0)$ to $(n-k,k)$ that hit the line
$y=x+u$, or equivalently,
 \[\cF^u = \{ F \in \binom{[n]}{k}:  
|F\cap[j]|\geq(j+u)/2 \text{ for some }j \}. \] 
Let $\tF^u, \hF^u$, and $\dhF^u$ be defined as before, but with respect to this new $\cF^u$. 
Similarly, we now use $\cF^t_i$ to mean $\cF^t_i(n,k)$. 
(One can think of this redefinition as applying the function $\first$ to everything in the
 previous definitions of the families.)
 
\subsection{Proof of Proposition \ref{main-k-prop}: Setup}

Let $k \geq t \geq 14$ and $n \geq (t+1)k$ be given ($\eta\in(0,1]$ will be given later).
The following basic inequalities will be used frequently without referring to explicitly.
\begin{align*}
&\frac kn\leq\frac1{t+1},\quad
\frac{n-k}n\geq\frac t{t+1},\quad
\frac n{n-k}\leq\frac {t+1}t,\\
&\bfrac n{n-k}^t\leq\(1+\frac 1t\)^t<e,\quad
\frac k{n-k}\leq\frac1{t},\quad
\frac {k(n-k)}{n^2}\leq\frac t{(t+1)^2}.
\end{align*}

Let $\cA$ and $\cB$ be non-empty shifted cross $t$-intersecting families 
in $\binom{[n]}{k}$.
Let $u=\lambda(\cA)$ and $v=\lambda(\cB)$. 
By Lemma~\ref{eq:>=2t} we have $u+v\geq 2t$.
If $u + v \geq 2t+1$, then (i) of Lemma~\ref{lem:k-hitline} gives
\[
|\cA||\cB|\leq\binom n{k-u} \binom n{k-v}\leq\binom n{k-t}\binom n{k-t-1}
<_t \binom{n-t}{k-t}^2,
\]
which gives (i) of the proposition. In fact, the last inequality can be
shown as follows:
\begin{align*}
 \lefteqn{\binom n{k-t}\binom n{k-t-1}\binom{n-t}{k-t}^{-2}}\nonumber\\
&=\frac{n\cdots(n-t+1)}{(n-k+t)\cdots(n-k+1)}
\frac{n\cdots(n-t+1)(k-t)}{(n-k+t+1)\cdots(n-k+1)}\nonumber\\
&=\left(\frac{n\cdots(n-t+1)}{(n-k+t)\cdots(n-k+1)}\right)^2
\frac n{n-k+t+1}\frac{k-t}n\nonumber\\
&<\left(\frac{n-t+1}{n-k+1}\right)^{2t+1}\frac1{t+1}
<\left(1+\frac1t\right)^{2t+1}\frac1{t+1}<\frac{e^{2+1/t}}{t+1}<_t 1,
 \end{align*}
where the last inequality holds for $t\geq 8$.

So we may assume that $u+v = 2t$, and by symmetry that $u \leq v$. 
For later use, we also notice that $\frac{e^{2+1/t}}{t+1}<_t \frac12$
for $t\geq 15$, while
$\left(1+\frac1t\right)^{2t+1}\frac1{t+1}<\frac12$ is true even when $t=14$.
Thus we have
\begin{equation}\label{eq:<_t 1/2}
\binom n{k-t}\binom n{k-t-1} \binom{n-t}{k-t}^{-2}<_t \frac12
\end{equation}
for $t\geq 14$. 

We partition $\cA$ and $\cB$ into families $\hA$, $\dhA$, $\tA$ and 
 $\hB$, $\dhB$, $\tB$, as we do near the beginning of Section \ref{sec:setup}, 
(but relative to the $k$-uniform versions of $ \tF^u, \hF^u$, and $\dhF^u$.) 

If $\hA=\emptyset$, then $\cA=\dhA\cup\tA$. 
Using (iii) and (i) of Lemma~\ref{lem:k-hitline} we have
$|\dhA|\leq\binom n{k-u-1}$, $|\tA|\leq\binom n{k-u-1}$ 
and $|\cB|\leq\binom n{k-v}$. Thus we get
\[
|\cA||\cB|\leq2\binom n{k-u-1} \binom n{k-v}\leq 2\binom n{k-t}\binom n{k-t-1}
<_t \binom{n-t}{k-t}^2,
\]
where the last inequality follows from \eqref{eq:<_t 1/2}, and
this is one of the points we really need $t\geq 14$.
The same holds for the case when $\hB=\emptyset$.
Thus if $\hA=\emptyset$ or $\hB=\emptyset$ then (i) of the proposition holds.

From now on we assume that $\hA\neq\emptyset$ and $\hB\neq\emptyset$.
Then Lemma~\ref{lem:structure} holds in our $k$-uniform setting as well,
namely, there exist unique nonnegative integers $s$ and $s'$ such that
$s-s'=(v-u)/2$, $\cA_s:=\hA\cup\dhA\subset\cF_s^u$,
and $\cB_{s'}:=\hB\cup\dhB\subset\cF_{s'}^v$.
It then follows from $\emptyset\neq\hB\subset\cF_{s'}^v$ that
\[
 k\geq v+s'=v+\(s-\frac{v-u}2\)=\frac{u+v}2+s=t+s.
\]
In summary, we may assume the following  {\bf setup}.
\begin{itemize}
\item $t\geq 14$, $s\geq s' \geq 0$, $k \geq t+s$, and $n \geq (t+1)k$. 
\item $u+v=2t$, $0\leq u\leq t\leq v\leq 2t$, and
$s-s'=(v-u)/2$. 
\item  $u=t-(s-s')$ and $v=t+(s-s')$.
\item $\cA=\hA\cup\dhA\cup\tA\subset\cF^u$, 
$\cB=\hB\cup\dhB\cup\tB\subset\cF^v$,
$\hA\neq\emptyset$ and $\hB\neq\emptyset$.
\item $\cA_s:=\hA\cup\dhA\subset\cF_s^u$ 
and $\cB_{s'}:=\hB\cup\dhB\subset\cF_{s'}^v$.
\end{itemize}
From here, our division into cases is the same as in the $p$-weight version.

\subsection{Proof of Proposition \ref{main-k-prop}: Easy cases}\label{sec:k-easy cases}
In this subsection, we prove the following.
\begin{lemma}\label{k-easycases}
If $s\geq 2$ then $\sqrt{|\cA||\cB|} < 0.89 \binom{n-t}{k-t}$.
\end{lemma}
\begin{proof}
Let $\bF^u_s:=(\hF^u\cup\dhF^u)\cap\cF^u_s$. 
Since $\cA=\tA\cup\cA_s$, $\tA\subset\tF^u$, 
$\cA_s\subset\bF^u_s$, we have $\cA\subset\tF^u\cup\bF^u_s$ and
\[
 |\cA|\leq|\tF^u|+|\bF^u_s|.
\]
By (i) of Lemma~\ref{lem:k-hitline} we have 
\[
|\tF^u|=\binom n{k-u-1}. 
\]
Since all walks in $\bF^u_s$ hit $(s,u+s)$, by counting the number of
walks from $(0,0)$ to $(s,u+s)$, and from $(s,u+s)$ to $(n-k,k)$, we get
\[
 |\bF^u_s|\leq\binom{u+2s}s\binom{n-u-2s}{k-u-s}.
\]
Thus we have
\[
 |\cA|\leq (a_1+a_2)\binom{n-u}{k-u},
\]
where
\begin{align*}
 a_1:=f(n,k,u,s)&:=\binom n{k-u-1}\binom{n-u}{k-u}^{-1},\\
 a_2:=g(n,k,u,s)&:=\binom{u+2s}s\binom{n-u-2s}{k-u-s}\binom{n-u}{k-u}^{-1}.
\end{align*}
This rather generous estimation is enough for our purpose 
(if $s\geq 2$ and $t\geq 14$) as we will see. In the same way
we have
\[
 |\cB|\leq (b_1+b_2)\binom{n-v}{k-v},
\]
where $b_1=f(n,k,v,s')$ and $b_2=g(n,k,v,s')$.
Notice that 
\[
\binom{n-u}{k-u}\binom{n-v}{k-v} \leq \cdots \leq 
\binom{n-(t-1)}{k-(t-1)}\binom{n-(t+1)}{k-(t+1)}\leq \binom{n-t}{k-t}^2.  
\]
Thus, to prove the lemma, it is enough to show that
 \[ (a_1 + a_2)(b_1 + b_2) < 0.89. \]
First we find bounds on the individual components. 
 
  \begin{claim}
    For $s \geq 2$, we have $a_1 < 0.195$, $b_1 < 0.528$, and $a_1b_1 < 0.038$. 
    For $s = 2$, we have $b_1 < 0.224$. 
  \end{claim}
  \begin{proof}
  Using $n \geq (t+1)k$ we have 
  \begin{align*}
a_1 
&= \frac{n\cdots(n-u+1)(k-u)}{(n-k+u+1)(n-k+u)\cdots(n-k+1)} \\
& \leq\bfrac{n-u+1}{n-k+1}^u \frac{k-u}{n - k + u + 1}  \\
& < \bfrac{(t+1)k+1}{tk+1}^u \frac k{kt}
 < \bfrac{t+1}{t}^u \frac{1}{t} 
                < \frac{e^{u/t}}{t}. 
   \end{align*}
The RHS is decreasing in $t$ and increasing in $u$. 
So $e^{u/t}/t$ is maximized when $u=t$ (recall that $u\leq t$). 
Using also $t\geq 14$ we have 
\[
a_1<\frac et\leq \frac e{14}<0.195.
\]
In the same way we have
\[
 b_1=f(n,k,v,s')<\frac{e^{v/t}}t.
\]
Since $v\leq 2t$, the RHS is maximized when $v=2t$, and we get
\[
b_1<\frac{e^{v/t}}t\leq\frac{e^2}{t}\leq\frac{e^2}{14}<0.528 
\]
in general. Further, when $s = 2$, we have $v=t+s-s'\leq t+2$, and
\[
b_1<\frac{e^{v/t}}t\leq\frac{e^{1+\frac 2t}}t\leq\frac{e^{\frac{16}{14}}}{14}
<0.224.
\]
Also we have
\[
 a_1b_1<\frac{e^{u/t}}t\frac{e^{v/t}}t=\frac{e^{(u+v)/t}}{t^2}=\frac{e^2}{t^2}
\leq\frac{e^2}{14^2}<0.038.
\]
 \end{proof}

  \begin{claim}
   For $s \geq 3$, $a_2 < 0.34$, $b_2 < 2.21$ and $a_2b_2 < 0.12$. 
   For $s = 2$,    $a_2 < 0.68$, $b_2 < 1.14$, and $a_2b_2 <0.47$. 
 \end{claim}    
 \begin{proof}
  Using that $n \geq (t+1)k$ and that $k \geq t + s\geq u+s$ we have
  \begin{align*}
    a_2\binom{u+2s}{s}^{-1} & = 
\frac{(k-u)\cdots(k-u-s+1)\,(n-k)\cdots (n-k-s+1)}{(n-u) \cdots (n-u-2s+1) } \\
               & < \bfrac{k-u}{n-u}^s \bfrac{n-k}{n-u-s}^s 
\leq \bfrac{k}{n}^s \leq \bfrac{1}{t+1}^s,
  \end{align*}
and
\[
 a_2=g(n,k,u,s)<\binom{u+2s}{s} \bfrac{1}{t+1}^s =: h(t,u,s)
\]
Similarly, noting that $k\geq v+s'$, we have
\[
 b_2=g(n,k,v,s')< h(t,v,s').
\]

We check that $h(t,u,s)$ is decreasing in $s$ for $s\geq 2$. In fact
\[
h(t,u,s)>h(t,u,s+1)
\] 
is (after some computation) equivalent to
   \[ s^2(t-3) + s(tu + 2t -3u -4) + (tu - u^2 + t -2u -1) > 0. \]
Considering the LHS as a quadric of $s$, it is minimized at
$s=-(\frac{u+2}2+\frac1{t-3})<0$. So the LHS is increasing in $s$ for $s\geq 2$,
and it suffices to check the
above inequality at $s=2$, that is,
\[
 3tu+9t-u^2-8u-21>0.
\]
This is certainly true for $u=0$. If $u\geq 1$, then, using $t\geq u$, 
the LHS satisfies
\[
 u(t-u)+8(t-u)+(2u+1)t-21\geq 3t-21>0,
\]
which verifies that $h(t,u,s)$ is decreasing in $s$.

Thus if $s\geq 3$, then, noting that $h(t,u,3)$ is increasing in $u$, we have
\[
 h(t,u,s)\leq h(t,u,3)\leq h(t,t,3)=\binom{t+6}3\bfrac 1{t+1}^3.
\]
The derivative of the RHS is
$-\frac{(2 t+11) (3 t+13)}{3(t+1)^4}<0$, and $h(t,t,3)$ is decreasing in $t$.
Consequently, if $s\geq 3$ and $t\geq 14$, then
\[
 a_2<h(14,14,3)<0.34.
\]
Similarly, if $s=2$ and $t\geq 14$, then
\[
 a_2<h(14,14,2)=0.68.
\]
 
Since $b_2=h(t,v,s')$ is increasing in $v$ and $v\leq 2t$, we have
$b_2<h(t,2t,s')$. For $s'=0,1$, we have $h(t,2t,0)=1$ and $h(t,2t,1)=2$.
Now let $s'\geq 2$. Then $h(t,2t,s')$ is decreasing in $t$. In fact, we have
\begin{align*}
 \frac{h(t,2t,s')}{h(t+1,2(t+1),s')}&=
\frac{(2t+s'+2)(2t+s'+1)}{(2t+2s'+2)(2t+2s'+1)}\({1+\frac 1{t+1}}\)^{s'}\\
&> \frac{(2t+s'+2)(2t+s'+1)}{(2t+2s'+2)(2t+2s'+1)}\({1+\frac {s'}{t+1}}\)\\
&= 1+\frac{s'(s'-1)}{2(t+1)(2s+2s'+1)}>1.
\end{align*}
Thus, for $s'\geq 2$ and $t\geq 14$, we have
\[
 h(t,2t,s')\leq h(14,28,s')=\binom{28+2s'}{s'}\frac 1{15^{s'}},
\]
where the RHS is decreasing in $s'$, and $h(14,28,s')\leq h(14,28,2)<2.21$.
Consequently, for $s'\geq 0$ and $t\geq 14$, we get
\[
 b_2=h(h,v,s')\leq h(14,28,2)<2.21.
\]

If $s=2$, then we replace $v\leq 2t$ with $v=t+s-s'\leq t+2$, and we get
$b_2=h(t,v,s')\leq h(t,t+2,s')$. Since $0\leq s'\leq s$,
by computing $h(t,t+2,s')$ for $s'=0,1,2$, it turns out that the maximum
is taken when $s'=1$. Namely, for $s=2$ and $t\geq 14$, we have
\[
 b_2\leq h(t,t+2,1)=\frac{t+4}{t+1}\leq h(14,16,1)=1.2.
\]
   
Finally, using that $u+2s=v+2s'=t+s+s'$, we have
\begin{align*}
a_2b_2&=h(t,u,s)\,h(t,v,s')=\binom{t+s+s'}s\binom{t+s+s'}{s'}\bfrac 1{t+1}^{s+s'}\\
&\leq\binom{t+s+s'}{(s+s')/2}^2\bfrac 1{t+1}^{s+s'}
=h(t,t,{\ts \frac{s+s'}2})^2\leq h(t,t,s)^2.
\end{align*}
Computing this for
$s = 2$ and $s = 3$ gives the required bounds on $a_2b_2$. 
\end{proof}   
Now, in the case that $s \geq 3$ we have 
\begin{align*}
 (a_1+a_2)(b_1+b_2)&=a_1b_1+a_1b_2+a_2b_1+a_2b_2\\
&<0.038+0.195\cdot 2.21+0.34\cdot 0.528+0.12<0.77,
\end{align*}
and when $s = 2$ we have 
\[a_1b_1+a_1b_2+a_2b_1+a_2b_2
< 0.038 + 0.195\cdot 1.14 + 0.68\cdot 0.224 + 0.47 < 0.89, \]  
which completes the proof of Lemma~\ref{k-easycases}.
\end{proof}

\subsection{Proof of Proposition~\ref{main-k-prop}: A harder case}\label{sec:k-non-diagonal}

\begin{lemma}\label{lem:(s,s')=(1,0)-k-unif} For $(s,s') = (1,0)$, we have $\sqrt{|\cA||\cB|} <_t \binom{n-t}{k-t}$.  \end{lemma}
 \begin{proof}
   Setting $(s,s') = (1,0)$ yields that $u = t-1$ and $v = t+1$.  We again follow the proof of Lemma \ref{lem:(s,s')=(1,0)}, redefining the constructions of that proof by applying the $\first$ operation to them. 

That is, let us define 
$D^{\cA}_i\in \hF^{t-1}\cap\cF_1^{t-1}$ ($1\leq i\leq n-2k+t-1$) and
$D^{\cB}_j\in \hF^{t+1}\cap\cF_0^{t+1}$ ($1\leq j\leq n-2k+t+1$) by
\begin{align*}
D^{\cA}_i&:= \first( [t-2]\cup \{t, t+1\}\cup \{t+1+i+2\ell\in [n] : \ell=1,2,\dots\}),\\
D^{\cB}_j&:= \first([t+1]\cup \{t+1+j+2\ell \in [n] : \ell=1,2,\dots\} ).
\end{align*}
Since
$\emptyset\neq\hA\subset\cA_1$ 
and $A\shiftsto D^{\cA}_1$ for any $A\in\hA$,
we have $D^{\cA}_1\in \cA$ and $\{ i : D^{\cA}_i\in \cA\}\neq\emptyset$.
Similarly, $\{j:D^{\cB}_j\in\cB\}\neq\emptyset$.
So the following values are well defined:
\[
I:=\max\{ i : D^{\cA}_i\in \cA\},\quad
J:=\max\{j:D^{\cB}_j\in\cB\}.
\]

The argument below is almost the same as in Subsection~\ref{sec:non-diagonal}.
The only difference is that all walks considered here are from $(0,0)$
to $(n-k,k)$. So we use Lemma~\ref{lem:k-hitline} instead of 
Lemma~\ref{lem:hitline}. In each case we will show that
\[
|\cA||\cB|=(|\tA|+|\cA_1|)(|\tB|+|\cB_0|)<_t\binom{n-t}{k-t}^2.  
\]

\smallskip\noindent{\bf Case 1.} When $I\geq 2$ and $J\geq 2$.

We divide walks in $\tA$ into three types (1a), (1b), and (1c) as in
Case 1 of Subsection~\ref{sec:non-diagonal}. 
The number of walks of (1a) is $\binom n{k-t-J+1}$ by 
(i) of Lemma~\ref{lem:k-hitline}. For (1b) use (ii) of 
Lemma~\ref{lem:k-hitline} and we get $\binom{n-t}{k-t}-\binom{n-t}{k-t-J+1}$.
Similarly, with aid of Example~\ref{ex:hitpoint}, we get 
$t\big(\binom{n-t-1}{k-t-1}-\binom{n-t-1}{k-t-J}\big)$ for (1c).
Thus we have
\begin{align*}
|\tA|&\leq\binom n{k-t-J+1}+\(\binom{n-t}{k-t}-\binom{n-t}{k-t-J+1}\)\\
&\quad+t\(\binom{n-t-1}{k-t-1}-\binom{n-t-1}{k-t-J}\)\\
&\leq\binom n{k-t-1}+\binom{n-t}{k-t}+t\binom{n-t-1}{k-t-1}.
\end{align*}
As for $\cA_1\subset\bF_1^{t-1}$ we notice that all walks in
$\bF_1^{t-1}$ hit $(1,t)$ without hitting $(0,t)$, and then 
from $(1,t)$ they never hit $y=x+1$. This gives
\[
 |\cA_1|\leq t\(\binom{n-t-1}{k-t}-\binom{n-t-1}{k-t-1}\).
\]

Any walk in $\tB$ hits the line $y=x+(t+1+I)$, or hits both 
$(0,t+1)$ and $y=x+t+2$ (see Subsection~\ref{sec:non-diagonal} for details).
Thus we get
\[
 |\tB|\leq\binom n{k-t-1-I}+\binom {n-t-1}{k-t-2} 
\leq\binom n{k-t-3}+\binom {n-t-1}{k-t-2}.
\]
Any walk in $\cB_0\subset\bF^{t+1}_0$ 
hits $(0,t+1)$ but does not hit the line $y=x+t+2$. This gives
\[
 |\cB_0|\leq\binom {n-t-1}{k-t-1}-\binom {n-t-1}{k-t-2}.
\]

In summary, we get
\begin{align*}
 |\cA|&\leq\binom n{k-t-1}+\binom{n-t}{k-t}+t\binom{n-t-1}{k-t},\\
 |\cB|&\leq\binom n{k-t-3}+\binom {n-t-1}{k-t-1}.
\end{align*}
Then
\begin{align*}
 |\cA|\binom{n-t}{k-t}^{-1}
&\leq\frac{n\cdots(n-t+1)(k-t)}{(n-k+t+1)\cdots(n-k+2)(n-k+1)}
+1+t\frac{n-k}{n-t}\\
&<\left(\frac n{n-k}\right)^t\frac k{n-k}+1+t\frac{n-k}{n-t},\\
|\cB|\binom{n-t}{k-t}^{-1}
&\leq\frac{n\cdots(n-t+1)(k-t)(k-t-1)(k-t-2)}{(n-k+t+3)\cdots(n-k+1)}
+\frac{k-t}{n-t}\\
&<\left(\frac n{n-k}\right)^t\left(\frac k{n-k}\right)^2\frac{k-t-2}{n-k+1}
+\frac{k-t}{n-t}.
\end{align*}
We also use
\begin{align*}
&\frac{k-t-2}{n-k+1}<\frac k{n-k},\quad
\frac{k-t}{n-t}<\frac kn,\\
&\frac{n-k}{n-t}\frac{k-t-2}{n-k+1}<\frac{n-k}n \frac k{n-k},\quad
 \frac{n-k}{n-t}\frac{k-t}{n-t}<\frac {n-k}n\frac kn.
\end{align*}
Then
\begin{align*}
\lefteqn{|\cA||\cB|\binom{n-t}{k-t}^{-2}} \\
&\leq
\(\left(\frac n{n-k}\right)^t\frac k{n-k}+1+t\frac{n-k}{n}\)
\(\left(\frac n{n-k}\right)^t\left(\frac k{n-k}\right)^3
+\frac{k}{n}\)\\
&=
\(\left(\frac n{n-k}\right)^t\frac n{n-k}\bfrac kn^2+\frac kn+
t\frac{k(n-k)}{n^2}\)
\(\left(\frac n{n-k}\right)^t\left(\frac n{n-k}\right)^3\bfrac kn^2
+1\)\\
&< \(\frac{e}{t(t+1)}+\frac{1}{t+1}+\frac{t^2}{(t+1)^2}\)
\(e\frac{t+1}{t^3}+1\)
\end{align*}
(For the first inequality, we remark that we did not estimate 
$|\cA|$ and $|\cB|$ separately. Instead,
we estimated each term appeared in the expansion of $|\cA||\cB|$ first,
then we factorized the sum of the terms afterwards.)
The RHS is equal to the $g(t)$ from \eqref{case1:g}, and thus
$<_t1$ for $t\geq 7$.

\smallskip\noindent{\bf Case 2.} When $I=1$.

Since walks in $\tA$ hit $y=x+t$, and walks in $\cB$ hit $y=x+t+1$, we get
\[
|\tA|\leq\binom n{k-t}\text{ and }
|\cB|\leq\binom n{k-t-1}.
\]

As for $\cA_1$ we look at the walks that hit $(1,t)$ without hitting
$y=x+t$. Among them, we delete the walks that hit all of 
$(1,t-2)$, $(1,t)$, $(3,t)$, and do not hit $y=x+t-2$ after hitting $(3,t)$.
(Here we use the fact that $D_2^{\cA}\not\in \cA$.) Thus we get
\begin{align*}
|\cA_1|&\leq t\(\binom{n-t-1}{k-t}-\binom{n-t-1}{k-t-1}\)
-(t-1)\(\binom{n-t-3}{k-t}-\binom{n-t-3}{k-t-1}\),
\end{align*}
For $\tA$ and $\cB$ we simply use the following estimation.
\begin{align*}
|\tA|\binom{n-t}{k-t}^{-1}&\leq\frac{n\cdots(n-t+1)}{(n-k+t)\cdots(n-k+1)}
\leq\left(\frac{n-t+1}{n-k+1}\right)^t<\left(1+\frac1t\right)^t,\\
|\cB|\binom{n-t}{k-t}^{-1}
&\leq\frac{n\cdots(n-t+1)}{(n-k+t)\cdots(n-k+1)}\frac{k-t}{n-k+t+1}
<\left(1+\frac1t\right)^t\frac 1t.
\end{align*}
For $\cA_1$ we need to estimate
\begin{align*}
|\cA_1|\binom{n-t}{k-t}^{-1}
&\leq t\left(\frac{n-k}{n-t}\left(1-\frac{k-t}{n-k}\right)\right)\\
&\quad -(t-1)\left(\frac{(n-k)(n-k-1)(n-k-2)}{(n-t)(n-t-1)}
\left(1-\frac{k-t}{n-k-2}\right)\right)\\
&=t\frac{n-2k+t}{n-t}-(t-1)\frac{(n-k)(n-k-1)(n-2k+t+2)}{(n-t)(n-t-1)(n-t-2)}\\
&<\(t\frac{n-2k}n-(t-1)\frac{(n-k)^2(n-2k)}{n^3}\)
+t\(\frac{n-2k+t}{n-t}-\frac{n-2k}n\).
\end{align*}
Let $p=\frac kn\leq\frac1{t+1}$. Then the first term of the RHS is
\[
 t(1-2p)-(t-1)(1-p)^2(1-2p)=(1-\alpha)(tq-(t-1)q^3),
\]
  where $q=1-p$ and $\alpha=\frac pq$, and thus we can reuse \eqref{eq:3t+1}.
For the second term we note that
\[
 \frac{n-2k+t}{n-t}-\frac{n-2k}{n}\leq\frac{n-2k+2t}n-\frac{n-2k}n=\frac{2t}n.
\]
Consequently we get
\[
|\cA_1|\binom{n-t}{k-t}^{-1}<\frac{t(t-1)(3t+1)}{(t+1)^3}+\frac{2t^2}n.
\]
Finally we use  $g(t)$ from \eqref{def:g(t)},
$g_2(t)$ from \eqref{def:g2(t)}, and note that 
$n\geq(t+1)k\geq (t+1)t$ to obtain
\begin{align}
 |\cA||\cB|\binom{n-t}{k-t}^{-2}&<
g(t)+\frac{2t^2}n\(1+\frac1t\)^t\frac1t < g(t) +\frac{2e}{t+1} \label{g19}\\
&< g_2(t)+\frac{2e}{t+1}\label{g20}.
\end{align}
The RHS \eqref{g20} is decreasing in $t$ for $t > 0$ and is less than  $1$ when $t=20$, 
and using $g(19)$ instead of $g_2(19)$ we get that 
$g(19)+\frac{2e}{19+1}<1$. 
This means that $|\cA||\cB|<_t\binom{n-t}{k-t}^2$ for $t\geq 19$.
For the remaining cases $14\leq t\leq 18$ we check 
$|\cA||\cB|<_t\binom{n-t}{k-t}^2$ by brute force as follows.
If $n>n_0(t):=\frac{2t}{1-g(t)}(1+\frac1t)^t$ 
then the RHS of \eqref{g19} is still $<1$. For smaller $n\leq n_0(t)$ we use
the trivial bounds for $|\tA|, |\cA_1|$ and $|\cB|$ given above to get:  
\[
\textstyle
|\cA||\cB|<
\left\{
\binom n{k-t}+t\(\binom{n-t-1}{k-t}-\binom{n-t-1}{k-t-1}\)
-(t-1)\(\binom{n-t-3}{k-t}-\binom{n-t-3}{k-t-1}\)
\right\}
\binom n{k-t-1},
\]
and check that the RHS is less than $1$ for all
$t\leq k, (t+1)k\leq n\leq n_0(t)$ with the aid of computer.
For example, in the case $t=14$, we have $\lfloor n_0(14)\rfloor =1023$,
and we compute the RHS of the above inequality 
for all $k$ and $n$ with $14\leq k, 15k\leq n\leq 1023$.
The cases $15\leq t\leq 18$ are similar and easier.
In the end, it turns out that $|\cA||\cB|<_t\binom{n-t}{k-t}^2$ for all 
$k\geq t\geq 14$, $n\geq (t+1)k$ in Case 2.

\smallskip\noindent{\bf Case 3.} When $J=1$.

Using the same reasoning as in Subsection~\ref{sec:non-diagonal}, we get
\begin{align*}
 |\cA|&\leq\binom n{k-t+1},\\
|\tB|&\leq\binom n{k-t-2},\\
|\cB_0|&\leq \binom{n-t-1}{k-t-1}-\binom{n-t-1}{k-t-2}
-\(\binom{n-t-3}{k-t-1}-\binom{n-t-3}{k-t-2}\).
\end{align*}
We continue to bound as follows:
\begin{align*}
|\cA|\binom{n-t}{k-t}^{-1}&\leq 
\frac{n\cdots(n-t+2)(n-t+1)}{(n-k+t-1)\cdots(n-k+1)(k-t+1)}\\
&\leq\(\frac{n-t+2}{n-k+1}\)^{t-1}\frac{n-t+1}{k-t+1}
<\left(1+\frac1t\right)^{t-1}\frac{n-t}{k-t},\\
|\tB|\binom{n-t}{k-t}^{-1}
&\leq\frac{n\cdots(n-t+1)}{(n-k+t+2)\cdots(n-k+3)}
\frac{(k-t)(k-t-1)}{(n-k+2)(n-k+1)}\\
&\leq\(\frac{n-t+1}{n-k+3}\)^t\frac{(k-t)(k-t-1)}{(n-k+2)(n-k+1)}
<\left(1+\frac1t\right)^t\bfrac{k-t}{n-k}^2\\
|\cB_0|\binom{n-t}{k-t}^{-1}
&\leq\frac{k-t}{n-t}\(
\left(1-\frac{k-t-1}{n-k+1}\right)
-\frac{(n-k)(n-k-1)}{(n-t-1)(n-t-2)}
\left(1-\frac{k-t-1}{n-k-1}\right)\)\\
&<\frac{k-t}{n-t}
\(1-\left(\frac{n-k}n\right)^2
+\frac {k(n-k)}{n^2}\right).
\end{align*}
(For simplicity we just threw away the first $-\frac{k-t-1}{n-k+1}$
from the last inequality, while this term was used in the $p$-weight version.)
Finally we get
\begin{align}
|\cA||\cB|\binom{n-t} {k-t}^{-2} &<
\left(1+\frac1t\right)^{t-1}
\(\left(1+\frac1t\right)^t\frac{nk}{(n-k)^2}
+1-\left(\frac{n-k}n\right)^2
+\frac {k(n-k)}{n^2}\)\nonumber\\
&\leq
\left(1+\frac1t\right)^{t-1}
\(\left(1+\frac1t\right)^t\frac{1+t}{t^2}
+1-\bfrac t{t+1}^2+\frac t{(t+1)^2}
\)\label{case3:14,15}\\
&<e \left(\frac{e
   (t+1)}{t^2}+\frac{3
   t+1}{(t+1)^2}\right).\nonumber
\end{align}
The RHS is decreasing in $t$, and less than $1$ when $t=16$. 
Further, \eqref{case3:14,15} is less than $1$ when $t=14,15$.
Thus $|\cA||\cB|<_t\binom{n-t}{k-t}^2$ for $t\geq 14$.
This completes the proof of Case 3, and so of  
Lemma~\ref{lem:(s,s')=(1,0)-k-unif}.
\end{proof} 
 
\subsection{Proof of Proposition~\ref{main-k-prop}: extremal cases}\label{sec:k-diagonal}

This is that case that $s = s' \in \{0,1\}$. Let $s\in\{0,1\}$ and let
$$
D_i'=\first(D_i)=\first\([1,t-1]\cup\{t+s, t+2s\}
\cup\{t+2s+i+2j:j\geq 1\}\)
$$
for $1\leq i\leq k-t-s=:i_{\kmax}$. 
Notice that 
$$D'_{i_{\kmax}}=\first\([1,t-1]\cup\{t+s, t+2s\}\cup\{k+s+2j:j\geq 1\}\)$$ 
and 
$$D'_{i_{\kmax}-1}=\first\([1,t-1]\cup\{t+s, t+2s\}\cup\{k+s+2j-1:j\geq 1\}\).$$
For any $A \in \hA\neq\emptyset$ 
it is easy to check that $A \shiftsto D'_1$, 
and hence $D'_1 \in \cA$. Similarly $D'_1 \in \cB$.
Let $I':=\max\{ i : D'_i\in \cA\}$ and $J':=\max\{ j : D'_j\in \cB\}$.

\begin{claim}\label{I neq ikmax}
If $I'\neq i_{\kmax}$, then
there is $\beta=\beta(t)>0$ such that 
\begin{align*}
|\cB\setminus\cF^t_s(n,k)|&\leq (1-\beta) |\cF^t_s(n,k)\setminus\cA|.
\end{align*}
\end{claim}

\begin{proof} 
First we show that 
\begin{equation}\label{eq:F-A k} 
|\cF^t_s(n,k)\setminus \cA|\geq {t \choose s}
\binom{n-2s-t-I'}{k-s-t}\frac{n-2k+t-I'}{n-2s-t-I'}.
\end{equation}
and
\begin{equation}\label{eq:B-A k}
|\cB\setminus \cF^t_s(n,k)|\leq \binom{n}{k-t-I'}.
\end{equation}

Consider a walk $W$ that hits $(s,s+t)$ and satisfies $W \shiftsto D_{I'+1}$. 
Since $D'_{I'+1}\not\in \cA$ we have $W\in\cF^t_s(n,k) \setminus \cA$.
Also $W$ must hit $Q_1=(s,t-s)$ and $Q_2=(s+I'+1,s+t)$. There are
$\binom ts$ ways for $W$ to go from $(0,0)$ to $Q_1$, then the next $2s+I'+1$
steps to $Q_2$ are unique.
From $Q_2$ the walk must not hit $y = x + (t - I')$. 
The number of such walks is equal to the number of walks from $(0,0)$
to $(x_0,y_0)$ that hit $y=x+c$ where $x_0=(n-k)-(s+I'+1)$,
$y_0=k-(s+t)$, and $c=1$. 
So we can count this number using (ii) of Lemma~\ref{lem:k-hitline} as follows:
\[
 \binom{n-2s-t-I'-1}{k-s-t}- \binom{n-2s-t-I'-1}{k-s-t-1}=
\binom{n-2s-t-I'}{k-s-t}\frac{n-2k+t-I'}{n-2s-t-I'}. 
\]
 Thus the number of walks in $\cF^t_s(n,k) \setminus \cA$ is at least 
the RHS of \eqref{eq:F-A k}. 

 Next we show~\eqref{eq:B-A k}. Since $\first(\dual_t(D'_{I'}))\not\in\cB$, 
each walk in $\cB$ hits at least one of $(0,t+s)$, $(s,t+s)$, and $y=x+(t+I')$. 
Since each walk that hits $(0,t+s)$ or $(s,t+s)$ is in $\cF^t_s(n,k)$, 
each walk in $\cB\setminus \cF^t_s(n,k)$ hits $y=x+(t+I')$. 
This yields~\eqref{eq:B-A k}.

Now we consider a lower bound for $|\cF^t_s(n,k)\setminus \cA|$ based on~\eqref{eq:F-A k}.
We have
$$\frac{n-2k+t-{I'}}{n-t-2s-{I'}}>
\frac{n-3k-s+1}{n-k-s+1}>\frac{n-3k}{n-k}>\frac{(t+1)k-3k}{(t+1)k-k}=\frac{t-2}{t}.$$
We also have
\begin{align*}
&\binom{n-t-2s-{I'}}{k-t-s}\Big/\binom n{k-t-s} 
>\(\frac{n-k-s-{I'}}{n-k+t+s}\)^{k-t-s}\\
&=\(1+\frac{t+2s+I'}{n-k-s-{I'}}\)^{-(k-t-s)}
>\(1+\frac{t+2+{I'}}{(t-1)(k-t-s)}\)^{-(k-t-s)}>e^{-\frac{t+2+{I'}}{t-1}}.
\end{align*}
Thus we infer
$$
\text{( RHS of \eqref{eq:F-A k} )}/\binom n{k-t-s} >
t^s\frac{t-2}{t}e^{-\frac{t+2+{I'}}{t-1}}.
$$

Finally we consider an upper bound of $|\cB\setminus\cF^t_s(n,k)|$ based on~\eqref{eq:B-A k}.
We have
$$
\binom{n}{k-t-{I'}}\Big/\binom n{k-t-s}
\leq \(\frac{k-t-s}{n-k+t+{I'}}\)^{I'-s}\leq \(\frac 1t\)^{I'-s}.
$$
Therefore, it suffices to show that
$$
t^s\frac{t-2}{t}e^{-\frac{t+2+{I'}}{t-1}}>_t \(\frac 1t\)^{I'-s},
\text{ or }
f(t,i):=
\frac{t-2}{t}e^{-\frac{t+2+i}{t-1}}t^i >_t 1.
$$
By direct computation, we have
$\frac{\partial f}{\partial t}>0$ for $t\geq 2$. Further we have
$\frac{\partial f(8,i)}{\partial i}>0$ for $i\geq 1$, and $f(8,1)>1.2$. Hence,
 $f(t,i)>_t 1$ for every $t\geq 8$ and $i\geq 1$.
\end{proof}

Let
$f=|\cF_s^t(n,k)|$,
$a=|\cA|$,
$a_0=|\cA\cap\cF_s^t(n,k)|$,
$a_1=|\cA\bigtriangleup\cF_s^t(n,k)|$,
$a_f=|\cA\setminus\cF_s^t(n,k)|$, and
$f_a=|\cF_s^t(n,k)\setminus\cA|$.
Define $b,b_0,b_1,b_f,f_b$ similarly.
The proof of the next lemma is identical to that of Lemma~\ref{lemma:not-extremal}
(use Claim~\ref{I neq ikmax} in place of Claim~\ref{I neq imax}).

\begin{lemma}\label{lemma:not-extremal k}
Let $\eta>0$ be given.
If $I'\neq i_{\kmax}$, then one of the following holds.
\begin{enumerate}
\item $\sqrt{ab}<(1-\frac{\beta\eta}4)f$, 
where $\beta\in(0,1]$ depends only on $t$.
\item $a_1+b_1<\eta f$ and $\sqrt{ab}< f$.
\end{enumerate}
\end{lemma}

Finally we finish the proof of Proposition~\ref{main-k-prop}.
If $I\neq i_{\kmax}$, then one of (i) or (ii) of Proposition~\ref{main-k-prop}
holds by Lemma~\ref{lemma:not-extremal k}.
(In this case we always have $\sqrt{ab}<f$.)
The same holds for the case $J\neq i_{\kmax}$.

Consequently we may assume that $I'=J'=i_{\kmax}$. 
Since $I'=i_{\kmax}$ we have $D'_{i_{\kmax}}\in \cA$.
Then $C:=[k+s+1]\setminus\{t+s, t+2s\}\not\in\cB$ because
$|D'_{i_{\kmax}}\cap C|=t-1$.
Thus all walks $B$ in $\cB$ satisfy $ B\not\shiftsto C$, 
and $\cB\subset\cF_s^t(n,k)$ follows.
Similarly, $J'=i_{\kmax}$ yields $\cA\subset\cF_s^t(n,k)$. 
Thus we have $ab\leq f^2$ 
with equality holding iff $\cA=\cB=\cF_s^t(n,k)$.
Now we show that one of (i) or (ii) of 
Proposition~\ref{main-k-prop} holds.
Let $f_a=\xi_a f$, $f_b=\xi_b f$, and let $\xi=\xi_a+\xi_b$.
Then $a_1+b_1=f_a+f_b=\xi f$. On the other hand it follows that
$\sqrt{ab}=\sqrt{a_0b_0}=\sqrt{(1-\xi_a)(1-\xi_b)}f
\leq\frac{(1-\xi_1)+(1-\xi_b)}2 f=(1-\frac{\xi}2)f
\leq(1-\frac{\xi}2)\binom{n-t}{k-t}$.
Let $\eta$ be given. If $\xi<\eta$, then (ii) holds.
If $\xi\geq\eta$, then (i) holds by taking $\gamma^*$ 
slightly smaller than $1/2$.
This completes the proof of Proposition~\ref{main-k-prop}.
\qed

\begin{proof}[Proof of Theorem~\ref{k-thm}]
This follows from Proposition~\ref{main-prop} if $\cA$ and $\cB$ are shifted.
(Recall that if $n>(t+1)k$ then $|\cF_0^t(n,k)|>|\cF_1^t(n,k)|$.)
If they are not shifted, then
let $\cA'$ and $\cB'$ be shifted families we get from shifting $\cA$ and $\cB$.
Then the result holds for $\cA'$ and $\cB'$.
By Lemma~\ref{k-shifting} the same is true of $\cA$ and $\cB$, 
yielding the theorem. 
\end{proof}

\begin{proof}[Proof of Theorem \ref{k-thm-stability}]
This follows from Proposition~\ref{main-k-prop} unless
(ii) of Proposition~\ref{main-k-prop} happens with $s=1$.
In this last case, we have 
$\sqrt{|\cA||\cB|}\leq|\cF_1^t(n,k)|$.
Let $p:=k/n$. We will show that
\begin{equation}\label{f1/f0}
 |\cF_1^t(n,k)|/ |\cF_0^t(n,k)|<(t+2)p(1-p)+p^2=:g(p).
\end{equation}
Then, as in the proof of Theorem~\ref{p-thm-stability}, 
we get (i) of Theorem~\ref{k-thm-stability} by choosing $\gamma$ so that 
$g(\frac1{t+1+\delta})=1-\gamma\eta$, and this completes the proof.

Now noting that the LHS of \eqref{f1/f0} is
\[
 \frac{k-t}{(n-t)(n-t-1)}\big((t+2)(n-k)-(k-t-1)\big),
\]
we can rearrange \eqref{f1/f0} as follows:
\[
 f(p):=\big((t+2)-p(t+1)\big)n^2-(t+1)(t+2)n+t(t+1)^2>0.
\]
Since $p\leq\frac1{t+1+\delta}$ we have
$f(p)>f(\frac1{t+1})$. Then $f(\frac1{t+1})>0$ is equivalent to
$n^2-(t+2)n+t(t+1)>0$, which is certainly true for $n\geq(t+1)k\geq t(t+1)$.
\end{proof}

We proved Proposition~\ref{main-prop} for $p\leq\frac 1{t+1}$, but our proof
works for $p\leq\frac 1{t+1-\epsilon}$ as well, 
where $\epsilon>0$ is a sufficiently small constant depending $t$ only.
To see this we just notice that the functions used to bound the $p$-weights 
of families are continuous as functions of $p$. 
(This is not surprising. In fact it seems very likely that 
Proposition~\ref{main-prop} holds for $p\leq\frac 2{t+3+\delta}$, where
$\delta>0$ is any given constant.)
In the same way, one can verify that 
Proposition~\ref{main-k-prop} is true for $n\geq(t+1-\epsilon)k$ as well, 
where $\epsilon>0$ is a sufficiently small constant depending $t$ only.
Thus the upper bound for $|\cA||\cB|$ in
Theorem~\ref{k-thm} is also true even if
we replace the condition $n\geq(t+1)k$ with $n\geq(t+1-\epsilon)k$.
If $k$ is sufficiently large for fixed $t$, then 
$(t+1)(k-t+1)>(t+1-\epsilon)k$. Namely we have the following.
\begin{theorem}\label{k-thm k is big}
For every $t\geq 14$ there is some $k_0$ such that for every $k>k_0$ and
$n\geq (t+1)(k-t+1)$ we have the following.
If $\cA\subset\binom{[n]}k$ and $\cB\subset\binom{[n]}k$ are cross
$t$-intersecting, then
$$
|\cA||\cB|\leq\binom{n-t}{k-t}^2
$$
with equality holding iff $\cA=\cB\cong\cF_0^t(n,k)$,
or $n=(t+1)(k-t+1)$ and $\cA=\cB\cong\cF_1^t(n,k)$.
\end{theorem}

\section{An Application to Integer Sequences}\label{sec:application}
As an application of Theorem~\ref{p-thm} we consider
families of $t$-intersecting integer sequences, see e.g., \cite{FFintseq}.
Let $n,m,t$ be positive integers with $m\geq 2$ and $n\geq t$.
Then $\cH\subset [m]^n$ is considered to be a family of integer
sequences $(a_1,\ldots,a_n)$, $1\leq a_i\leq m$. We say that
$\cH$ is $t$-intersecting if any two sequences intersect in at least
$t$ positions, more precisely, $\#\{i:a_i={b}_i\}\geq t$ holds for
all $(a_1,\ldots,a_n),\,(b_1,\ldots,b_n)\in\cH$. 
To relate a family of sequences with a family of subsets, let us define
an obvious surjection $\sigma:[m]^n\to 2^{[n]}$ by
$\sigma((a_1,\ldots,a_n))=\{i:a_i=1\}$. Then 
$$\cH_i^t(n):=\{a\in[m]^n:\sigma(a)\in\cF_i^t(n)\}$$
is a $t$-intersecting family of integer sequences of size
$$|\cH_i^t(n)|=m^n\mu_{\frac 1m}(\cF_i^t(n)).$$
It is known from \cite{AK-p,FT,BE} that
if $r=\lfloor\frac{t-1}{m-2}\rfloor$,  
$n\geq t+2r$, and $\cH\subset[m]^n$ is a family of $t$-intersecting 
integer sequences, then
\begin{equation}\label{AKFT}
|\cH|\leq |\cH_r^t(n)|. 
\end{equation}
Observe that $|\cH_0^t(n)|=m^{n-t}$.
We extend \eqref{AKFT} in the case of $r=0$
to cross $t$-intersecting families of integer sequences.
We say that $\cA,\cB\subset [m]^n$ are cross $t$-intersecting
if $\#\{i:a_i={b}_i\}\geq t$ for
all $(a_1,\ldots,a_n)\in\cA$ and $(b_1,\ldots,b_n)\in\cB$. Two such families are called isomorphic, denoted 
$\cA\cong\cB$, if there are permutations $f_1,\ldots,f_n$ of $[m]$ 
and a permutation $g$ of $[n]$ such that 
$$\{(f_1(a_1),\ldots,f_n(a_n)):(a_1,\ldots,a_n)\in\cA\}
=\{(b_{g(1)},\ldots,b_{g(n)}):(b_1,\ldots,b_n)\in\cB\}.$$

Using Theorem~\ref{p-thm} we prove a conjecture posed 
in \cite{Teigen2} as follows.

\begin{theorem}\label{int-sec-thm}
Let $t\geq 14$, $m\geq t+1$ and $n\geq t$. 
If $\cA$ and $\cB$ are cross $t$-intersecting families
of integer sequences in $[m]^n$, then
$|\cA||\cB|\leq(m^{n-t})^2$.  Equality holds iff 
either $\cA=\cB\cong\cH_0^t(n)$, or
$m=t+1$ and $\cA=\cB\cong\cH^t_1(n)$.
\end{theorem}


To prove Theorem~\ref{int-sec-thm} we need some more preparation.
For $\cH\subset [m]^n$, $j\in[n]$ and $c\in[m]$, define another 
shifting operation $S_j^c(\cH)=\{S_j^c(a):a\in\cH\}\subset[m]^n$ as follows.
For $a=(a_1,\ldots,a_n)$ let $S_j(a_1,\ldots,a_n):=(b_1,\ldots,b_n)$ where
$b_\ell=a_\ell$ for $\ell\in[n]\setminus\{j\}$ and $b_j=1$. Then let
$S_j^c(a)=S_j(a)$ if $a_j=c$ and $S_j(a)\not\in\cH$, otherwise let
$S_j^c(a)=a$.
Namely, by $S_j^c(a)$, we replace $a_j$ with $1$ if $a_j=c$, but we do
this replacement only if the resulting sequence is not in the original family
$\cH$. We say that $\cH$ is shifted if
$S_j^c(\cH)=\cH$ for all $j\in[n]$ and $c\in[m]$.

\begin{lemma}\label{lem:intseq}
For $\cA,\cB\subset[m]^n$, $j,t\in[n]$, and $c\in[m]$, 
we have the following.
\begin{enumerate}
\item $|S_j^c(\cA)|=|\cA|$.
\item If $\cA$ and $\cB$ are cross $t$-intersecting families, then 
$S_j^c(\cA)$ and $S_j^c(\cB)$ are cross $t$-intersecting families as well.
\item Starting from $\cA$ and $\cB$ we obtain shifted families of sequences
by repeatedly shifting two families simultaneously finitely many times.
\item Let $m\geq 3$, and let $\ell$ be chosen so that
$\max_i|\cH_i^t(n)|=|\cH_\ell^t(n)|$.
If $\cA$ and $\cB$ are cross $t$-intersecting families with
$S_j^c(\cA)=S_j^c(\cB)=\cH_\ell^t(n)$,
then $\cA=\cB\cong\cH_\ell^t(n)$.
\item
If $\cA$ and $\cB$ are shifted cross $t$-intersecting,
then $\sigma(\cA)$ and $\sigma(\cB)$ are cross $t$-intersecting 
families of subsets in $2^{[n]}$.
\end{enumerate}
\end{lemma}
One can prove the above (i)--(iv) similarly as the proof of 
Lemmas~\ref{lem:shifting} and \ref{k-shifting}. 
See \cite{Teigen2} for the proof of (v).
We mention that (ii) is due to Kleitman \cite{K}, and 
(v) is observed by Frankl and F\"uredi \cite{FFintseq}. 

\begin{proof}[Proof of Theorem~\ref{int-sec-thm}]
Let $\cA$ and $\cB$ be cross $t$-intersecting families in $[m]^n$,
and let $\cA'$ and $\cB'$ be corresponding shifted families guaranteed
by Lemma~\ref{lem:intseq}.
By letting $\cF:=\sigma(\cA)\subset 2^{[n]}$ we have
\begin{equation}\label{eq:|H|}
|\cA|=|\cA'|\leq \sum_{x\in\cF}(m-1)^{n-|x|}=m^n\mu_{\frac 1m}(\cF).
\end{equation}
Similarly $|\cB|=m^n\mu_{\frac 1m}(\cG)$, where $\cG:=\sigma(\cB)$.
Since $\cF$ and $\cG$ are cross $t$-intersecting families 
it follows from Theorem~\ref{p-thm} that
\begin{equation}\label{eq:mu(F1)mu(F2)}
 \mu_{\frac 1m}(\cF) \mu_{\frac 1m}(\cG)\leq (1/m)^{2t}.
\end{equation}
By \eqref{eq:|H|} and \eqref{eq:mu(F1)mu(F2)} we have
$$
 |\cA||\cB|\leq (m^n)^2\mu_{\frac 1m}(\cF) \mu_{\frac 1m}(\cG)
\leq (m^n)^2(1/m)^{2t}=(m^{n-t})^2.
$$

Now suppose that $|\cA||\cB|=(m^{n-t})^2$. Then we need equality
in \eqref{eq:mu(F1)mu(F2)}. By Theorem~\ref{p-thm} we have
$\cF=\cG\cong\cF_0^t(n)$, or $m=t+1$ and $\cF=\cG\cong\cF_1^t(n)$.
We also need equality in \eqref{eq:|H|}. By the definition of $\cF$ and $\cG$
we have $\cA'_1=\cB'_2\cong\cH_0^t(n)$, or $m=t+1$ and
$\cA'_1=\cB'_2\cong\cH_1^t(n)$. By this together with Lemma~\ref{lem:intseq}
(v) we can conclude that $\cA=\cB\cong\cH_0^t(n)$, or $m=t+1$ and
$\cA=\cB\cong\cH_1^t(n)$. This completes the proof of 
Theorem~\ref{int-sec-thm}.
\end{proof}

\section*{Acknowledgment}
The authors thank the anonymous referees for their valuable comments.
The authors also thank Hajime Tanaka for telling us that Moon \cite{M}
proved Theorem~\ref{int-sec-thm} for all $t\geq 2$ and $m\geq t+2$.

\end{document}